\documentclass{conm-p-l}

\usepackage{amssymb,amsmath,amsthm,xspace,mathrsfs,enumerate}
\usepackage[breaklinks=true]{hyperref}%To get clickable links to papers
\usepackage[msc-links]{amsrefs}

\newcommand{\m}{\ensuremath{\mathfrak{m}}\xspace}

\newcommand{\vi}{\ensuremath{\varepsilon}\xspace}
\newcommand{\bk}{\ensuremath{\Bbbk}\xspace}
\newcommand{\Z}{\ensuremath{\mathbb{Z}}\xspace}
\newcommand{\hb}{\ensuremath{\mathscr{H}_{\lambda,\kappa}}\xspace}
\newcommand{\scrh}{\ensuremath{\mathscr{H}}\xspace}
\newcommand{\one}[1]{\ensuremath{{#1}_{(1)}}\xspace}
\newcommand{\two}[1]{\ensuremath{{#1}_{(2)}}\xspace}
\newcommand{\three}[1]{\ensuremath{{#1}_{(3)}}\xspace}
\newcommand{\lev}[1]{\ensuremath{\mathscr{L}_{\leqslant #1}}\xspace}
\newcommand{\tangle}[1]{\ensuremath{\langle #1 \rangle}\xspace}

\DeclareMathOperator{\id}{\ensuremath{id}\xspace}
\DeclareMathOperator{\ad}{\ensuremath{ad}\xspace}
\DeclareMathOperator{\ch}{\ensuremath{char}\xspace}
\DeclareMathOperator{\im}{\ensuremath{im}\xspace}
\DeclareMathOperator{\Sym}{\ensuremath{Sym}\xspace}
\DeclareMathOperator{\Hom}{\ensuremath{Hom}\xspace}
\DeclareMathOperator{\End}{\ensuremath{End}\xspace}
\DeclareMathOperator{\Prim}{\ensuremath{Prim}\xspace}
\DeclareMathOperator{\codim}{\ensuremath{codim}\xspace}
\DeclareMathOperator{\fix}{\ensuremath{Fix}\xspace}
\DeclareMathOperator{\rad}{\ensuremath{Rad}\xspace}
\DeclareMathOperator{\Ext}{\ensuremath{Ext}\xspace}

\newtheorem{theorem}{Theorem}[section]
\newtheorem{lemma}[theorem]{Lemma}
\newtheorem{prop}[theorem]{Proposition}
\newtheorem{cor}[theorem]{Corollary}

\theoremstyle{definition}
\newtheorem{definition}[theorem]{Definition}

\newtheorem{stand}[equation]{Standing Assumption}

\theoremstyle{remark}
\newtheorem{remark}[theorem]{Remark}

\numberwithin{equation}{section}

\begin{document}
\title[Nil-Coxeter algebras, cocommutative algebras, PBW
property]{Generalized nil-Coxeter algebras, cocommutative algebras, and
the PBW property}

\author{Apoorva Khare}

\address{Departments of Mathematics and Statistics, Stanford University,
Stanford, CA - 94305}

\email{khare@stanford.edu}



%\thanks{Revised version, resubmitted on \today}

\begin{abstract}
Poincar\'e--Birkhoff--Witt (PBW) Theorems have attracted significant
attention since the work of Drinfeld (1986), Lusztig (1989), and
Etingof--Ginzburg (2002) on deformations of skew group algebras $H \ltimes
\Sym(V)$, as well as for other cocommutative Hopf algebras $H$.
In this paper we show that such PBW theorems do not require the full Hopf
algebra structure, by working in the more general setting of a
``cocommutative algebra'', which involves a coproduct but not a counit or
antipode. Special cases include infinitesimal Hecke algebras, as well as
symplectic reflection algebras, rational Cherednik algebras, and more
generally, Drinfeld orbifold algebras. In this generality we identify
precise conditions that are equivalent to the PBW property, including a
Yetter--Drinfeld type compatibility condition and a Jacobi identity. We
also characterize the graded deformations that possess the PBW property.
In turn, the PBW property helps identify an analogue of symplectic
reflections in general cocommutative bialgebras.

Next, we introduce a family of cocommutative algebras outside the
traditionally studied settings: generalized nil-Coxeter algebras. These
are necessarily not Hopf algebras, in fact, not even (weak) bialgebras.
For the corresponding family of deformed smash product algebras, we
compute the center as well as abelianization, and classify all simple
modules.
\end{abstract}
\maketitle

\setcounter{page}{139}

%{{{1 Section 1 - Introduction
\section{Introduction}

In the study of deformation algebras, their structure and
representations, one commonly begins by understanding their connection to
the corresponding associated graded algebras (which are generally better
behaved). Such connections of course provide desirable ``monomial
bases'', but also additional structural and representation-theoretic
knowledge.

A first step in understanding these connections involves showing that
these filtered algebras satisfy the \textit{Poincar\'e--Birkhoff--Witt
(PBW) property}, in that they are isomorphic as vector spaces to their
associated graded algebras.
Such results are known as \textit{PBW theorems} in the literature. The
terminology of course originates with the classical result for the
universal enveloping algebra of a Lie algebra. However, it has gathered
renewed attention over the past few decades owing to tremendous interest
in the study of orbifold algebras and their generalizations, which we now
briefly describe.

In a seminal paper \cite{Dr}, Drinfeld pioneered the study of smash
product algebras of the form $\bk G \ltimes \Sym(V)$, where a group $G$
acts on a $\bk$-vector space $V$. Drinfeld's results were rediscovered
and extended by Etingof and Ginzburg in their landmark paper \cite{EG},
which introduced symplectic reflection algebras and furthered our
understanding of rational Cherednik algebras. These algebras serve as
``non-commutative'' coordinate rings of the orbifolds $V/G$; see
\cite{Lu} for a related setting.
Subsequently, Etingof, Ginzburg, and Gan replaced the group by algebraic
distributions of a reductive Lie group $G$. This led to the study of
infinitesimal Hecke algebras in \cite{EGG} (and several recent papers),
where $U \mathfrak{g}$ acts on $\Sym(V)$, with $\mathfrak{g} = {\rm
Lie}(G)$. These families of deformed algebras continue to be popular and
important objects of study, with connections to representation theory,
combinatorics, and mathematical physics.

A common theme underlying all of these settings is that a cocommutative
Hopf algebra $H$ acts on the vector space $V$ and hence on $\Sym(V)$. The
aforementioned families of algebras $\hb$ are created by deforming two
sets of relations:
\begin{itemize}
\item The relations $V \wedge V \mapsto 0$ in the smash product algebra
$H \ltimes \Sym(V)$ are deformed using an anti-symmetric bilinear form
$\kappa : V \wedge V \to H$, or more generally, $\kappa : V \wedge V \to
H \oplus V$. These deformed relations feature in \cites{Dr,EG,EGG}, and
follow-up works.

\item The relations $g \cdot v = g(v) g$ for grouplike elements $g$ with
$H$ a group algebra, were deformed by Lusztig \cite{Lu} to create graded
affine Hecke algebras, using a bilinear form $\lambda : H \otimes V \to
H$.
\end{itemize}

The forms $\lambda, \kappa$ define a filtered algebra, and an important
question is to characterize those deformations $\hb$ whose associated
graded algebra is isomorphic to $\scrh_{0,0} = H \ltimes \Sym(V)$.
Such parameters $\lambda,\kappa$ are said to correspond to \textit{PBW
deformations}, and have been studied in the aforementioned works as well
as by Braverman and Gaitsgory \cite{BG} among others. More recently, in a
series of papers \cites{SW1,SW2,SW3}, Shepler and Witherspoon have shown
PBW theorems in a wide variety of settings (skew group algebras, Drinfeld
orbifold algebras, Drinfeld Hecke algebras, \dots), that encompass many
of the aforementioned cases. We also point the reader to the
comprehensive survey \cite{SW4} for more on the subject. This includes
the case of $\Sym(V)$ replaced by a quantum symmetric algebra. Perhaps
one of the most general versions in the literature is the recent work
\cite{WW1} by Walton and Witherspoon, in which $H$ is replaced by a Hopf
algebra, and $\Sym(V)$ by a Koszul algebra. For completeness, we also
mention work in related flavors: \cite{HOZ} studies generalized Koszul
algebras, while \cites{BaBe,WW2} analyze deformations of Hopf algebra
actions on ``doubled'' pairs of module algebras.\smallskip

We now point out some of the novel features and extensions in the present
paper. First, all of the aforementioned settings involve $H$ being a
bialgebra -- in fact, a Hopf algebra. In this paper we isolate the
structure required to study the PBW property, and show that it includes
the coproduct but not the counit or antipode. More precisely, we work in
the more general framework of a (cocommutative) algebra with coproduct.
This is a strictly weaker setting than that of a bialgebra, as it also
includes examples such as the nil-Coxeter (or nil-Hecke) algebra
associated to a Weyl group, $NC_W$. Recall that these algebras were
originally introduced by Fomin and Stanley \cite{FS} as Demazure
operators in the study of Schubert polynomials, though they appear
implicitly in previous work \cites{BGG,KK} on the cohomology of
generalized flag varieties for semisimple and Kac--Moody groups,
respectively; see also \cite{LS}. Nil-Coxeter algebras have subsequently
been studied in their own right \cites{Br,Y} as well as in the context of
categorification \cites{Kho,KL}, among others.

Nil-Coxeter algebras are necessarily not bialgebras (hence not Hopf
algebras). Thus, deformations over such cocommutative algebras have not
been considered to date in the literature.

Second, we introduce a novel class of Hecke-type algebras, the
\textit{generalized nil-Coxeter algebras}, which encompass the usual
nil-Coxeter algebras. These algebras have not been studied in the
literature. In this paper we will specifically study deformations over
generalized nil-Coxeter algebras. Moreover, our results are
characteristic-free.

An additional novelty of the present work is that in all of the
aforementioned works in the literature, either the bilinear form
$\kappa_V : V \wedge V \to V$ is assumed to be identically zero, or/and
$\lambda : H \otimes V \to H$ is identically zero. The present paper
addresses this gap by working with algebras for which all three
parameters $\lambda, \kappa_V, \kappa_A = \kappa - \kappa_V$ are allowed
to be nonzero. (All notation is explained in Definition \ref{Dhk} below.)

\subsection*{Organization of the paper}

We now outline the contents of the present paper, which can be thought of
as having two parts. In Section \ref{S2}, we introduce the general notion
of a \textit{cocommutative $\bk$-algebra} $A$, i.e., an algebra with a
multiplicative coproduct map that is cocommutative (over a unital ground
ring $\bk$). We next state and prove one of our main results: a PBW-type
theorem for deformations $\hb$ of the smash product algebra $\scrh_{0,0}
= A \ltimes \Sym(V)$. Here, $A$ acts on tensor powers of $V$ via the
coproduct, and on the symmetric algebra because of cocommutativity.

In Section \ref{S3}, we explain the connection between the PBW theorem
and deformation theory. Specifically, we identify the graded
$\bk[t]$-deformations of $\scrh_{0,0}$ whose fiber at $t=1$ has the PBW
property. This extends various results in the literature; see
\cites{SW1,SW3}.
The first part of the paper concludes in Section \ref{S4}, by examining
well-known notions in the Hopf algebra literature in the broader setting
of cocommutative algebras. This includes studying the cases where $A$ is
a cocommutative bialgebra or Hopf algebra. We classify the parameters
$\lambda, \kappa$ for which $\hb$ has the same structure, and relate the
PBW property to the Yetter--Drinfeld condition, a natural compatibility
condition that arises in Hopf-theoretic settings. We also extend the
notion of `symplectic reflections' from groups to all cocommutative
bialgebras.

In the second part of the paper, we study a specific family of
cocommutative algebras that are not yet fully explored in the literature.
Thus, in Section \ref{S5} we introduce a family of \textit{generalized
nil-Coxeter algebras} associated to a Coxeter group $W$; these are
closely related to Coxeter groups and their generalizations studied by
Coxeter and Shephard--Todd \cites{Cox,Co,ST}.

Generalized nil-Coxeter algebras are necessarily not bialgebras;
thus they fall strictly outside the Hopf-theoretic setting. In the
remainder of the paper, we study the deformations $\hb$ over generalized
nil-Coxeter algebras. We first study the Jacobi identity in such algebras
$\hb$, and classify all Drinfeld-type deformations $\scrh_{0, \kappa}$
with the PBW property. In the final section of the paper, we study
additional properties of the algebras $\hb$, including computing the
center and abelianization, and classifying simple modules.
%}}}

\section{Cocommutative algebras, smash products, and the PBW
theorem}\label{S2}

\noindent \textbf{Global assumptions:}
Throughout this paper, we work over a ground ring $\bk$, which is a
unital commutative ring. 
 We also fix a cocommutative $\bk$-algebra $(A,\Delta)$, defined below,
and a $\bk$-free $A$-module $V$.

By $\dim V$ for a free $\bk$-module $V$, we will mean the (possibly
infinite) $\bk$-rank of $V$. In this paper, all $\bk$-modules, including
all $\bk$-algebras, are assumed to be $\bk$-free. Unless otherwise
specified, all (Hopf) algebras, modules, and bases of modules are with
respect to $\bk$, and all tensor products are over $\bk$.

%{{{1 Section 2.1 - Cocommutative algebras and the PBW theorem
\subsection{Cocommutative algebras and the PBW theorem}

We begin by introducing the main construction of interest and the main
result of the first part of this paper.

\begin{definition}
Suppose $A$ is a unital associative $\bk$-algebra.
\begin{enumerate}
\item $A$ is an \textit{algebra with coproduct} if there exists a
$\bk$-algebra map $\Delta : A \to A \otimes_\bk A$ called the
\textit{coproduct}, such that $\Delta(1) = 1 \otimes 1$ and $\Delta$ is
coassociative, i.e., $(\Delta \otimes 1) \circ \Delta = (1 \otimes
\Delta) \circ \Delta : A \to A \otimes A \otimes A$.

\item An algebra with coproduct is said to be \textit{cocommutative} if
$\Delta = \Delta^{op}$.
\end{enumerate}
\end{definition}

Notice that bialgebras and Hopf algebras (with the usual coproduct) are
examples of algebras with coproduct (with $\bk$ a field). As pointed out
to us by Susan Montgomery, one could \textit{a priori} have considered
weak bialgebras (these feature prominently in the theory of fusion
categories \cite{ENO}), but these provide no additional examples, as
explained at the end of \cite{BNS}*{\S 2.1}:
since $\Delta(1) = 1 \otimes 1$ by assumption,
a cocommutative algebra is a bi/Hopf-algebra
if and only if it is a weak bi/Hopf-algebra.
Additional examples do arise, however, using nil-Coxeter algebras, as
explained in Remark \ref{Rweak} below. These algebras show that the
notion of an algebra with coproduct is strictly weaker than that of a
(weak) bialgebra.

We also remark that every unital $\bk$-algebra $A$ is an algebra with
coproduct, if we define $\Delta_L(a) := a \otimes 1$ or $\Delta_R(a) := 1
\otimes a$. (Thus, the definition essentially involves a choice of
coproduct.) However, $A$ need not have a cocommutative coproduct in
general.

Given $a \in A$, write $\Delta(a) = \sum \one{a} \otimes \two{a}$ and
$\Delta^{op}(a) = \sum \two{a} \otimes \one{a}$, in the usual Sweedler
notation. We now use $\Delta$ to first define tensor and symmetric
product $A$-module algebras, as well as \textit{undeformed Drinfeld Hecke
algebras}.
Suppose $(A, \Delta)$ acts on a free $\bk$-module $V$ (not necessarily of
finite rank), denoted by $v \mapsto a(v)$. Notice that $TV := T_\bk V$
has an augmentation ideal $T^+V := V \cdot T_\bk V$, and this ideal is an
$A$-module algebra via:
\[
a (v_1 \otimes \cdots \otimes v_n) := \sum \one{a}(v_1) \otimes
\cdots \otimes a_{(n)}(v_n), \qquad \forall a \in A,\ v_1, \dots, v_n \in
V,\ n \geqslant 1.
\]

\noindent We do not include the case $n=0$ here, since $A$ does not have
a counit $\vi$.

\begin{definition}
Given a $\bk$-algebra $A$, let $A^{mult}$ denote the left $A$-module $A$,
under left multiplication.
Now given $(A,\Delta)$ and $V$ as above, the \textit{smash product} of
$TV$ and $A$, denoted by $TV \rtimes A^{mult}$, is defined to be the
$\bk$-algebra $T (V \oplus A^{mult})$, with the multiplication relations
given by $a \cdot a' := aa'$ in $A$,
$({\bf v'} \otimes a') \cdot (1 \otimes a) = {\bf v'} \otimes a'a$, and
\[
({\bf v'} \otimes a') \cdot({\bf v} \otimes a) := \sum ({\bf v'} \cdot
\one{a'}({\bf v})) \otimes \two{a'} \cdot a, \qquad \forall a,a' \in A,
{\bf v'} \in TV, \ {\bf v} \in T^+ V.
\]
\end{definition}

We use $- \rtimes A$ rather than $A \ltimes -$ in this paper. Also note
that for $1_A$ to commute with $V$ requires $\Delta(1) = 1 \otimes 1$ as
above.
Now denote by $\wedge^2 V \subset V \otimes_\bk V$ the $\bk$-span of $v
\wedge v' := v \otimes v' - v' \otimes v$; then $\wedge^2 V$ is an
$A$-submodule of $T^+ V$ because of the cocommutativity assumption on
$A$, which implies that $a(v_1 \wedge v_2) = \sum \one{a}(v_1) \wedge
\two{a}(v_2)$. Thus, one can quotient $TV \rtimes A$ by the related
two-sided ``$A$-module ideal'', to define:
\begin{equation}
\scrh_{0,0}(A,V) = \scrh_{0,0} = \Sym(V) \rtimes A := \frac{TV \rtimes
A}{(TV \cdot \wedge^2 V \cdot TV) \rtimes A}.
\end{equation}

\noindent The algebra $\scrh_{0,0}(A,V)$ will be referred to as the
\textit{smash product of $\Sym(V)$ and $A$}.
We are now able to introduce deformations of this smash product algebra.

\begin{definition}\label{Dhk}
Given $(A, \Delta)$ and $V$ as above, as well as bilinear forms $\lambda
\in \Hom_\bk(V \otimes A, A)$ and $\kappa \in \Hom_\bk(\wedge^2 V, A
\oplus V)$, the \textit{deformed smash product algebra $\hb = \hb(A,V)$
with parameters $\lambda, \kappa$} is defined to be the quotient of $T(V
\oplus A)$ by the multiplication in $A$ and by
\begin{equation}
a v - \sum \one{a}(v) \two{a} = \lambda(a,v), \quad v v' - v'v =: [v,v']
= \kappa(v,v'), \quad \forall a \in A,\ v,v' \in V.
\end{equation}

\noindent Also define $\kappa_V \in \Hom_\bk(V \wedge V, V)$ and
$\kappa_A \in \Hom_\bk(V \wedge V, A)$ to be the projections of $\kappa$
to $V,A$ respectively.
\end{definition}

\noindent Observe that $\lambda$ being trivial is equivalent to the
$A$-action preserving the grading on $\Sym(V)$.
Moreover, we will write $\hb$ instead of $\hb(A,V)$ if $A,V$ are clear
from context. 

The deformed smash product algebras $\hb = \hb(A,V)$ encompass a very
large family of deformations considered in the literature, including
universal enveloping algebras,
skew group algebras,
Drinfeld orbifold algebras,
Drinfeld Hecke algebras,
symplectic reflection algebras,
rational Cherednik algebras,
degenerate affine Hecke algebras and graded Hecke algebras,
Weyl algebras,
infinitesimal Hecke algebras,
and many others. This is an area of research that is the focus of
tremendous recent activity; see \cites{CBH, Dr, EGG, EG, Kh, KT, LT, Lu,
Ti, Ts}, and subsequent follow-up works in the literature.\smallskip

\begin{remark}
In order to place the work in context, we briefly comment on how our
framework compares to other papers in the PBW literature.
The paper encompasses other works in two aspects: first, the algebra
$(A,\Delta)$ is strictly weaker than a bialgebra. Second, the deformation
parameters $\lambda, \kappa_V, \kappa_A$ can all be nonzero.
At the same time, we impose two restrictions that are present in some
papers but not in others: first, we work with $\Sym(V)$ and not a quantum
algebra, nor a general Koszul algebra (e.g., a PBW algebra). Second, for
ease of exposition we only consider algebras with $\im (\kappa_V)$ a
subset of $V$ instead of $V \otimes A$; this is akin to the assumption
$\lambda \equiv 0$ in \cites{SW1,WW1}, or $\kappa_V \equiv 0$ in
\cite{SW3}.
\end{remark}

Notice that the algebras $\hb$ are filtered, by assigning $\deg A = 0,
\deg V = 1$. We say that $\hb$ has the \textbf{PBW property} if the
surjection from $\scrh_{0,0} = \Sym(V) \rtimes A$ to the associated
graded algebra of $\hb$ is an isomorphism.
Equivalently, the PBW theorem holds for $\hb$ if for any (totally)
ordered $\bk$-basis $\{ x_i : i \in I \}$ of the free $\bk$-module $V$
and $\{ a \in J_1 \}$ of the $\bk$-free $\bk$-algebra $A$, the collection
\[
\{ X \cdot a : X \text{ is a word in the } x_i \text{ in non-decreasing
order of subscripts, } a \in J_1 \}
\]
is a $\bk$-basis of $\hb$. We now state the main result of the first part
of the paper, which is a PBW Theorem for the algebras $\hb$.

\begin{theorem}[PBW Theorem]\label{Tpbw}
Suppose $(A, \Delta)$ is a $\bk$-free cocommutative $\bk$-algebra, and
$V$ a $\bk$-free $A$-module. Define $\hb$ with $\kappa = \kappa_V \oplus
\kappa_A : V \wedge V \to V \oplus A$ as above, and suppose $A = \bk
\cdot 1 \bigoplus A'$ for a free $\bk$-submodule $A' \subset A$. Then the
following are equivalent:
\begin{enumerate}
\item $\hb$ has the PBW property (for a $\bk$-basis of $V$ and a
$\bk$-basis of $A$ containing $1$).

\item The natural map $: A \oplus (V \otimes A) \to \hb$ is an injection.

\item $\lambda : A \otimes V \to A$ and $\kappa : V \wedge V \to V \oplus
A$ satisfy the following conditions:
\begin{enumerate}
\item {\em $A$-action on $V$:}
For all $a,a' \in A$ and $v \in V$, the following equation holds in $A$:
\begin{equation}\label{Eaction}
\lambda(aa', v) = a \lambda(a',v) + \sum \lambda(a, \one{a'}(v))
\two{a'}.
\end{equation}

\item {\em $A$-compatibility of $\lambda,\kappa$:}
For all $a \in A$ and $v,v' \in V$, the following equations hold in $A$
and $V \otimes A$ respectively:
\begin{align}
a \kappa_A(v,v') - &\ \sum \kappa_A(\one{a}(v), \two{a}(v')) \three{a}
\label{Ecompat1}\\
= &\ \lambda(\lambda(a,v),v') - \lambda(\lambda(a,v'),v) - \lambda(a,
\kappa_V(v,v')), \notag\\
\sum \one{a}(\kappa_V(v,v')) \two{a} - &\ \sum \kappa_V(\one{a}(v),
\two{a}(v')) \three{a}\label{Ecompat2}\\
= &\ \sum \one{\lambda(a,v)}(v') \two{\lambda(a,v)}
- \sum \one{\lambda(a,v')}(v) \two{\lambda(a,v')} \notag\\
&\ + \sum \one{a}(v) \lambda(\two{a},v') - \sum \one{a}(v')
\lambda(\two{a},v).\notag
\end{align}

\item {\em Jacobi identities:}
For all $v_1, v_2, v_3 \in V$, the following cyclic sum vanishes:
\[
\sum_\circlearrowright [\kappa(v_1,v_2),v_3] := [\kappa(v_1, v_2), v_3] +
[\kappa(v_2, v_3), v_1] + [\kappa(v_3, v_1), v_2] = 0.
\]

\noindent More precisely, the following equations hold in $A$ and $V
\otimes A$ respectively (identifying $V$ with $V \otimes 1_A \subset V
\otimes A$):
\begin{align}
\sum_\circlearrowright \lambda(\kappa_A(v_1, v_2), v_3) = &\
\sum_\circlearrowright \kappa_A(v_1, \kappa_V(v_2,
v_3)),\label{Ejacobi1}\\
\sum_\circlearrowright \kappa_V(\kappa_V(v_1, v_2), v_3) = &\
\sum_\circlearrowright v_1 \kappa_A(v_2, v_3) -
\sum_\circlearrowright \one{\kappa_A(v_1, v_2)}(v_3) \two{\kappa_A(v_1,
v_2)}. \label{Ejacobi2}
\end{align}
\end{enumerate}
\end{enumerate}
\end{theorem}

As observed by Shepler and Witherspoon in their papers
\cite{SW1}--\cite{SW4}, their versions of the PBW theorem, and therefore
ours, specialize to the PBW criteria for the algebras studied by
Drinfeld, Etingof--Ginzburg, Lusztig, as well as in numerous follow-up
papers on these families of algebras (see the remarks following
Definition \ref{Dhk} for additional references). Thus, Theorem \ref{Tpbw}
unifies several results in the literature and extends them to arbitrary
cocommutative algebras. As a specific example, we point the reader to
\cite{SW3}*{Theorem 3.1} for the analogous result with $\bk$ a field, $A$
a group algebra, and $\kappa_V \equiv 0$.

\begin{remark}\label{R23}
Notice that the conditions in part (3) of the theorem always hold in
$\hb$. In other words, Equations \eqref{Eaction}--\eqref{Ejacobi2} hold
in the image of the space $A \oplus (V \otimes A)$ in $\hb$, by
considering the equations corresponding to the associativity of the
algebra $\hb$:
\[
aa' \cdot v = a \cdot (a' \cdot v), \qquad
a \cdot (v v' - v' v) = a \cdot \kappa(v,v'), \qquad
\sum_\circlearrowright [\kappa(v_1,v_2),v_3] = 0.
\]
The assertion of Theorem \ref{Tpbw} is that the PBW property is
equivalent to these equations holding in $A \oplus (V \otimes A)$.
\end{remark}

\begin{remark}\label{Rdim2}
It is easy to verify that the Jacobi identities \eqref{Ejacobi1},
\eqref{Ejacobi2} hold in $A \oplus (V \otimes A)$ if $\dim_\bk V
\leqslant 2$, since in that case the left and right hand sides of both
equations vanish. If moreover $\dim_\bk V \leqslant 1$, then the
$A$-compatibility conditions \eqref{Ecompat1}, \eqref{Ecompat2} also hold
in $A \oplus (V \otimes A)$, since $\kappa_V, \kappa_A \equiv 0$.
\end{remark}
%}}}

%{{{1 Section 2.2 - Proof of the PBW Theorem
\subsection{Proof of the PBW Theorem}

We now prove Theorem \ref{Tpbw} using the Diamond Lemma \cite{Be}. As we
work with a general cocommutative algebra (which is strictly weaker than
a cocommutative bialgebra), and moreover, work with possibly nonzero
$\lambda, \kappa_V$, the proof is written out in some detail.
To prove Theorem \ref{Tpbw}, we will require the unit $1$ to be one of
our $\bk$-basis vectors for $A$; words involving this basis vector are to
be considered ``without'' the $1$.

\begin{proof}[Proof of the PBW Theorem \ref{Tpbw}]
Clearly, $(1) \implies (2)$, and $(2) \implies (3)$ using Remark
\ref{R23}. The goal in the remainder of this proof (and this section) is
to show that $(3) \implies (1)$. We begin by writing down the relations
in $\hb$ systematically. Recall that $A = \bk \cdot 1_A \oplus A'$; now
suppose $\{ a_j : j \in J \}$ is a $\bk$-basis of the $\bk$-submodule
$A'$. Write
\begin{equation}
J_1 := \{ a_j : j \in J \} \sqcup \{ 1_A \}, \quad a_0 := 1_A, \quad J_0
:= J \sqcup \{ 0 \}.
\end{equation}

\noindent We also fix a total ordering on $J_1$ and correspondingly on
$J_0$, with $0 \leqslant j$ for all $j \in J_0$.

Next, fix a totally ordered $\bk$-basis of $V$, denoted by $\{ x_i : i
\in I \}$. (Thus, $I$ is also totally ordered.) We then define various
structure constants, with the sums running over $J_0$ and $I$, and using
Einstein notation throughout. We first define the structure constants
from $A$ and its action on $V$:
\begin{equation}
a_j a_k = u_{jk}^l a_l, \qquad
a_j(x_k) = s_{jk}^h x_h, \qquad
\Delta(a_j) = r_j^{kl} a_k \otimes a_l.
\end{equation}

\noindent In particular, $u_{j0}^i = u_{0j}^i = \delta_{i,j}$, $s_{0k}^i
= \delta_{i,k}$, and $r_0^{kl} = \delta_{k,0} \delta_{l,0}$. Next, we
define the structure constants for the maps $\lambda, \kappa$:
\begin{equation}
\kappa_A(x_j,x_k) = v_{jk}^l a_l, \qquad \kappa_V(x_j, x_k) = w_{jk}^l
x_l, \qquad (j>k); \qquad \lambda(a_j,x_k) = q_{jk}^l a_l.
\end{equation}

\noindent It now follows that $\hb$ is a quotient of $T(V \oplus A)$,
with the defining relations:
\begin{align}\label{Ealgebra}
x_j x_k = &\ x_k x_j + v_{jk}^l a_l + w_{jk}^h x_h \quad (j > k),
\notag\\
a_j a_k = &\ u_{jk}^l a_l,\\
a_j x_k = &\ t_{jk}^{mn} x_n a_m + q_{jk}^l a_l, \quad \text{where }
t_{jk}^{mn} = r_j^{lm} s_{lk}^n.\notag
\end{align}

\noindent Thus, the $q,r,s,t,u,v,w$ are all structure constants in $\bk$,
for all choices of indices.

To show (1), we first write down additional consequences of the structure
of $A,V$. The following equations encode the associativity,
coassociativity, and cocommutativity of $A$:
\begin{align}
u_{jm}^l u_{il}^n = &\ u_{ij}^l u_{lm}^n \; \quad \forall i,j,m,n;
\notag\\
r_i^{jl} r_l^{mn} = &\ r_i^{kn} r_k^{jm} \quad \forall i,j,m,n;\\
r_j^{kl} = &\ r_j^{lk} \qquad \quad \forall j,k,l \in J_0 = J \sqcup \{
0 \}. \notag
\end{align}

\noindent The next condition is that $\Delta$ is multiplicative, which
yields:
\begin{align*}
u_{jk}^l r_l^{mn} (a_m \otimes a_n) = &\ u_{jk}^l \Delta(a_l) =
\Delta(a_j a_k) = \Delta(a_j) \Delta(a_k)\\
= &\ r_j^{cd} r_k^{ef} (a_c \otimes a_d) (a_e \otimes a_f) = r_j^{cd}
r_k^{ef} u_{ce}^m u_{df}^n (a_m \otimes a_n).
\end{align*}

\noindent Equating coefficients in $A \otimes A$, we conclude that
\begin{equation}\label{Ecoprod}
u_{jk}^l r_l^{mn}  = r_j^{cd} r_k^{ef} u_{ce}^m u_{df}^n.
\end{equation}

\noindent Finally, $V$ is an $A$-module, which yields:
\[
s_{jm}^n s_{ki}^m x_n = a_j \left( s_{ki}^m x_m \right) = a_j(a_k(x_i)) =
(a_j a_k)(x_i) = u_{jk}^m a_m(x_i) = u_{jk}^m s_{mi}^n x_n,
\]

\noindent whereby we get
\begin{equation}\label{Emodule}
s_{jm}^n s_{ki}^m = u_{jk}^m s_{mi}^n.
\end{equation}

We now proceed with the proof, using the terminology of \cite{Be}. The
\textit{reduction system} $S$ consists of the set of algebra relations
\eqref{Ealgebra}.
Then expressions in the left and right hand sides in the equations in
\eqref{Ealgebra} are what Bergman calls $f_\sigma$ and $W_\sigma$,
respectively.

Define $X := \{ a_j : j \in J \} \cup \{ x_i : i \in I \}$. Then the
expressions in the free semigroup $\tangle{X}$ generated by $X$ that are
irreducible (i.e., cannot be reduced via the operations $f_\sigma \mapsto
W_\sigma$ via the Equations \eqref{Ealgebra}) are precisely the PBW-basis
that was claimed earlier, i.e. words $x_{i_1} \cdots x_{i_l} \cdot a_j$,
for $j \in J_0$ and $i_1 \leqslant i_2 \leqslant \dots \leqslant i_l$,
all in $I$. This also includes the trivial word $1$.

Next, define a \textit{semigroup partial ordering} $\leqslant$ on $X$,
first on its generators via:
\begin{equation}
1 < x_i < a_j, \quad \forall j \in J,\ i \in I,
\end{equation}

\noindent and then extend to a total order on $\tangle{X}$, as follows:
words of length $m$ are strictly smaller than words of length $n$,
whenever $m<n$; and
words of equal lengths are (totally) ordered lexicographically.
It is easy to see that $\leqslant$ is a semigroup partial order on
$\tangle{X}$, i.e., if $a \leqslant b$ then $waw' \leqslant wbw'$ for all
$w,w' \in \tangle{X}$. Moreover, $\leqslant$ is indeed compatible with
$S$, in that each $f_\sigma$ reduces to a linear combination $W_\sigma$
of monomials strictly smaller than $f_\sigma$.

We now recall the \textit{descending chain condition}, which says that
given a monomial $B \in \tangle{X}$, any sequence of reductions applied
to $B$ yields an expression that is irreducible in finitely many steps.
Now the following result holds.

\begin{lemma}\label{Ldiamond}
The semigroup partial order $\leqslant$ on $\tangle{X}$ satisfies the
descending chain condition.
\end{lemma}

\begin{proof}
We prove a stronger assertion; namely, we produce an explicit upper bound
for the number of reductions successively applicable on a monomial. Given
a word $w = T_1 \cdots T_n$, with $T_i \in X\ \forall i$, define its
\textit{misordering index} $mis(w)$ to be $o + p + pr + q + r^3$, where
\begin{align*}
o = o(w) := &\ \# \{ (i,j) : i<j,\ T_i, T_j \in V,\ T_i > T_j \},\\
p = p(w) := &\ \# \{ (i,j) : i<j,\ T_i \in A',\ T_j \in V \},\\
q = q(w) := &\ \# \{ i : T_i \in A' \},\\
r = r(w) := &\ \# \{ i : T_i \in V \} = n-q.
\end{align*}

\noindent We now claim that each reduction strictly reduces the
misordering index of each resulting monomial; this claim shows the
result.
As an illustration of the claim, we present the most involved case: when
$f_\sigma = x_j x_k$ with $j>k$, and the monomial we consider via the
reduction $f_\sigma \mapsto W_\sigma$ corresponds to $a_l$ for some $l
\in J$. For this new word $w'$, notice that $q$ increases by 1, whereas
$r$ reduces by 2 (so $r \geqslant 2$), $o$ reduces by at least 1, and $p$
may increase by at most the number of $x$ to the right of the new $a$,
which is at most $r-2$. So, $o+q$ does not increase, and we now claim
that $p + pr + r^3$ strictly reduces.
Indeed, $p' \leqslant p+r-2,\ r' \leqslant r-2$, whence:
\begin{align*}
p'(1+r') + (r')^3 \leqslant &\ (p+r-2)(1+(r-2)) + (r-2)^3\\
\leqslant &\ p(1+r) + (r-2) + (r-2)^2 + (r-2)^3\\
= &\ p(1+r) + (r-2)(r^2 - 3r + 3) < p(1+r) + r \cdot r^2.
\end{align*}

\noindent Hence $mis(w') < mis(w)$ as desired.
\end{proof}

The final item utilized in the proof of the PBW theorem, is the notion of
ambiguities. It is clear that no $f_\sigma$ is a subset of $f_\tau$ for
some $\sigma,\tau \in S$; hence there are no \textit{inclusion
ambiguities}. In light of Lemma \ref{Ldiamond} and the Diamond Lemma
\cite{Be}*{Theorem 1.2}, it suffices to resolve all \textit{overlap
ambiguities} using the given conditions in (3). We begin by writing down
these conditions explicitly using the structure constants in $A$.
Explicit computations using these constants and Equations
\eqref{Eaction}--\eqref{Ejacobi2} yield the following five equations,
respectively:
\begin{align}
u_{jk}^l q_{li}^h = &\
q_{ki}^l u_{jl}^h + t_{ki}^{mn} q_{jn}^l u_{lm}^h, \label{Epbw1}\\
v_{jk}^l u_{il}^h - t_{ij}^{mn} t_{mk}^{cd} v_{nc}^l u_{ld}^h = &\
q_{ij}^l q_{lk}^h - q_{ik}^l q_{lj}^h - w_{jk}^l q_{il}^h,\\
w_{jk}^l t_{il}^{dc} - t_{ij}^{mn} t_{mk}^{dl} w_{nl}^c =&\
q_{ij}^m t_{mk}^{dc} - q_{ik}^m t_{mj}^{dc}
+ t_{ij}^{mc} q_{mk}^d - t_{ik}^{mc} q_{mj}^d,\\
\sum_{\circlearrowright (i,j,k)} v_{ij}^l q_{lk}^h = &\
\sum_{\circlearrowright (i,j,k)} w_{jk}^m v_{im}^h,\\
\sum_{\circlearrowright (i,j,k)} w_{ij}^l w_{lk}^h \cdot (x_h \otimes
a_0) =&
\sum_{\circlearrowright (i,j,k)} v_{jk}^m (x_i \otimes a_m) -
\sum_{\circlearrowright (i,j,k)} v_{ij}^l t_{lk}^{dc} \cdot (x_c \otimes
a_d).
\end{align}

We now resolve the overlap ambiguities, which are of four types, and
correspond to the associativity of the algebra $\hb$ (see Remark
\ref{R23}):
\[
a_i a_j a_k,\ a_j a_k x_i,\ a_k x_i x_j (i>j),\ x_i x_j x_k (i>j>k).
\]

\noindent Notice that the first type is resolvable because $A$ is an
associative algebra. We only analyse the second type of ambiguity in what
follows; the others involve carrying out similar (and more longwinded)
computations, that use the structure constants of the cocommutative algebra
$A$ with coproduct.

To resolve the ambiguity $a_j a_k x_i$, using the above analysis in this
proof we compute:
\[
(a_j a_k) x_i = u_{jk}^l a_l x_i = u_{jk}^l t_{li}^{mh} x_h a_m +
u_{jk}^l q_{li}^h a_h = u_{jk}^l r_l^{mc} s_{ci}^h \cdot x_h a_m +
u_{jk}^l q_{li}^h a_h.
\]

\noindent On the other hand,
\begin{align*}
a_j (a_k x_i) = t_{ki}^{fg} a_j x_g a_f + q_{ki}^l a_j a_l
= &\ t_{ki}^{fg} t_{jg}^{yh} x_h a_y a_f + t_{ki}^{fg} q_{jg}^l a_l a_f +
q_{ki}^l u_{jl}^h a_h\\
= &\ t_{ki}^{fg} t_{jg}^{yh} u_{yf}^l x_h a_l
+ t_{ki}^{mn} q_{jn}^l u_{lm}^h a_h + q_{ki}^l u_{jl}^h a_h.
\end{align*}

\noindent The overlap ambiguity is resolved if these two expressions are
shown to be equal. In light of \eqref{Epbw1}, it suffices to show that,
after relabelling indices, we have for all $i,j,k,l,h$ (or $h$-$l$):
\[
u_{jk}^m r_m^{lf} s_{fi}^h = t_{ki}^{fg} t_{jg}^{yh} u_{yf}^l.
\]

\noindent To see why this holds, begin with the right-hand side, expand
using the definition of $t$, and then use Equations \eqref{Ecoprod},
\eqref{Emodule} above:
\begin{align*}
t_{ki}^{fg} t_{jg}^{yh} u_{yf}^l = &\ r_k^{fa} s_{ai}^g \cdot r_j^{yn}
s_{ng}^h \cdot u_{yf}^l = r_k^{fa} r_j^{yn} u_{yf}^l \cdot (s_{ng}^h
s_{ai}^g)\\
= &\ r_k^{fa} r_j^{yn} u_{yf}^l (u_{na}^g s_{gi}^h) = r_j^{yn} r_k^{fa}
u_{yf}^l u_{na}^g \cdot s_{gi}^h\\
= &\ u_{jk}^m r_m^{lg} \cdot s_{gi}^h,
\end{align*}

\noindent which is precisely the left-hand side. Thus the ambiguity is
resolved.
\end{proof}
%}}}

%{{{1 Section 3 - Characterization via deformation theory
\section{Characterization via deformation theory}\label{S3}

We now explain how PBW deformations can be naturally understood via
deformation theory. In this section, suppose $\bk$ is a field. Given an
associative algebra $B$ and an indeterminate $t$, a \textit{deformation
of $B$ over $\bk[t]$} is an associative $\bk[t]$-algebra $(B_t,*)$ that
is isomorphic to $B[t]$ as a vector space, such that $B_t / t B_t$ is
isomorphic to $B$ as a $\bk$-algebra. In particular, we can write the
multiplication of two elements $b_1,b_2 \in B \otimes t^0 \subset B_t$
as:
\[
b_1 * b_2 = b_1 b_2 + \sum_{j > 0} \mu_j(b_1, b_2) t^j,
\]

\noindent where $\mu_j : B \otimes B \to B$ is $\bk$-linear and only
finitely many terms are nonzero in the above sum.

If moreover $B$ is $\Z^{\geqslant 0}$-graded, then a \textit{graded
$\bk[t]$-deformation of $B$} is a deformation of $B$ over $\bk[t]$ that
is graded with $\deg t = 1$, i.e., each $\mu_j : B \otimes B \to B$ is
homogeneous of degree $-j$. The map $\mu_j$ is also called the
\textit{$j$th multiplication map}.

Henceforth in this section we will consider the special case of the
$\Z^{\geqslant 0}$-graded algebra $B := \scrh_{0,0} = \Sym(V) \rtimes A$,
with $(A, \Delta)$ a cocommutative algebra as above. Our first goal in
this section is to show that the PBW property for the algebras $\hb$ has
a natural reformulation in terms of graded deformations of $\scrh_{0,0}$
over $\bk[t]$. Such a result was shown in \cite{SW3}*{\S 6} in the
special case of $A$ a group algebra, and further assuming that $\kappa_V
\equiv 0$. We now explain how the assumption $\kappa_V \equiv 0$ is
related to that in \textit{loc.~cit.} of requiring $V \otimes V \subset
\ker \mu_1$, by extending the result to general $\kappa_V : V \wedge V
\to V$ and all cocommutative algebras $A$.

\begin{theorem}\label{Tdeform}
Suppose $\bk$ is a field (of arbitrary characteristic), $(A,\Delta)$ is
cocommutative, and $V$ an $A$-module. Consider the following two
statements.
\begin{enumerate}
\item $\hb$ satisfies the PBW Theorem \ref{Tpbw}.

\item There exists a graded $\bk[t]$-deformation $B_t$ of $B :=
\scrh_{0,0} = \Sym(V) \rtimes A$, whose multiplication maps $\mu_1,
\mu_2$ satisfy (for all $v,v' \in V$ and $a \in A$):
\begin{align}\label{Edeform}
\lambda(a, v) = &\ \mu_1(a \otimes v) - \sum \mu_1(\one{a}(v) \otimes
\two{a}),\notag\\
\kappa_V(v, v') = &\ \mu_1(v \otimes v') - \mu_1(v' \otimes v),\\
\kappa_A(v, v') = &\ \mu_2(v \otimes v') - \mu_2(v' \otimes v).\notag
\end{align}
\end{enumerate}

\noindent Then $(1) \implies (2)$, and the converse holds if $\dim A,
\dim V$ are both finite. Moreover, if these statements hold then $\hb
\cong B_t|_{t=1}$.
\end{theorem}

\noindent Thus, the structure maps $\lambda, \kappa_V, \kappa_A$ in $\hb$
can be naturally reformulated using the multiplication maps $\mu_1,
\mu_2$ in a graded deformation of $\scrh_{0,0}$, whenever $\hb$ has the
PBW property.

\begin{proof}
We provide a sketch of the proof as it closely resembles the arguments
for proving \cite{SW3}*{Proposition 6.5 and Theorem 6.11}. First suppose
(1) holds. Define $(B_t, *)$ to be the associative algebra over $\bk[t]$
generated by $A,V$, with the following relations (for all $a \in A, v,v'
\in V$):
\begin{align*}
a * v = &\ \sum \one{a}(v) * \two{a} + \lambda(a,v) t,\\
v * v' - v' * v = &\ \kappa_V(v,v') t + \kappa_A(v,v') t^2.
\end{align*}

\noindent This yields a $\Z^{\geqslant 0}$-graded algebra with $\deg(t) =
\deg(V) = 1$ and $\deg(A) = 0$. Moreover, $B_t \cong \scrh_{0,0}
\otimes_\bk \bk[t]$ as vector spaces, since $\hb$ has the PBW property.
Now verify using the definitions and the relations in the algebra $(B_t,
*)$, that
\[
\kappa_V(v,v') t + \kappa_A(v,v') t^2 = v * v' - v' * v = vv' + \sum_{j >
0} \mu_j(v \otimes v') t^j - v'v - \sum_{j > 0} \mu_j(v' \otimes v) t^j.
\]

\noindent As this is an equality of polynomials in $\scrh_{0,0}[t]$, we
equate the linear and quadratic terms in $t$ on both sides, to obtain the
last two equations in \eqref{Edeform}.
The first equation in \eqref{Edeform} follows from a similar computation. 
This shows (2), and moreover, $B_t|_{t=1} \cong \hb$.

Conversely, suppose (2) holds, and $\dim V, \dim A < \infty$. Define $F_t
:= T_{\bk[t]}(V \oplus A) / (a \cdot a' - aa')$; then we have an algebra
map $f : F_t \to B_t$, which sends monomials $x_1 \cdots x_k$ (with each
$x_i \in V \oplus A$) to $x_1 * \cdots * x_k$. One shows as in \cite{SW3}
that $f$ is surjective, and the vectors
\[
a v - \sum \one{a}(v) \two{a} - \lambda(a,v) t = a v - \sum \one{a}(v)
\two{a} - \mu_1(a,v) t + \sum \mu_1(\one{a}(v) \otimes \two{a}) t
\]
and
\begin{align*}
&\ v v' - v' v - \kappa_V(v, v') t - \kappa_A(v, v') t^2\\
= &\ v v' - v' v - \mu_1(v, v') t + \mu_1(v', v) t - \mu_2(v, v') t^2 +
\mu_2(v', v) t^2
\end{align*}

\noindent lie in $\ker(f)$. We use here that $a * v = av + \mu_1(a
\otimes v) t$ and $v * v' = vv' + \mu_1(v \otimes v') t + \mu_2(v \otimes
v') t^2$, since $\deg \mu_j = -j$ for all $j>0$.

This analysis implies that $\scrh_{\lambda,\kappa,t} \twoheadrightarrow
B_t$ as $\Z^{\geqslant 0}$-graded $\bk$-algebras, where
$\scrh_{\lambda,\kappa,t}$ is the quotient of $F_t$ by the relations
\[
a v - \sum \one{a}(v) \two{a} - \lambda(a,v) t, \qquad
v v' - v' v - \kappa_V(v,v') t - \kappa_A(v,v') t^2.
\]

\noindent Now using that $A,V$ are finite-dimensional, verify that the
graded components of the two algebras satisfy: $\deg \scrh_{\lambda,
\kappa, t}[m] \leqslant \deg B_t[m]$. Hence the dimensions agree for each
$m$, whence $\scrh_{\lambda,\kappa,t} \cong B_t$. It follows that $\hb =
\scrh_{\lambda,\kappa,t} |_{t=1} \cong B_t |_{t=1}$ as filtered algebras.
Now as explained at the end of the proof of \cite{SW3}*{Theorem 6.11},
$\hb$ has the PBW property.
\end{proof}
%}}}

\section{The case of bialgebras and Hopf algebras}\label{S4}

In this section we study a special case of the general framework above,
but now requiring that $A$ is a cocommutative bialgebra (with counit
$\vi$), or Hopf algebra (with counit $\vi$ and antipode $S$).
This is indeed the case in a large number of prominent and well-studied
examples in the literature, as discussed after Definition \ref{Dhk}.

We begin by observing that the cocommutative algebra structure on $A$
automatically extends to $\scrh_{0,0} = \Sym(V) \rtimes A$, setting
$\Delta(v) = v \otimes 1 + 1 \otimes v$ for all $v \in  V$.
Akin to the usual Hopf-theoretic setting, we now introduce the following
notation.

\begin{definition}
Given a cocommutative algebra $(A,\Delta)$, an element $a \in A$ is said
to be \textit{primitive} (respectively, \textit{grouplike}), if
$\Delta(a) = 1 \otimes a + a \otimes 1$ (respectively, $\Delta(a) = a
\otimes a$).
\end{definition}

We now observe that it is possible to classify when the deformed algebra
$\hb$ is a cocommutative algebra, a bialgebra, or a Hopf algebra, under
the assumption that $A$ has the same structure and $V$ is primitive.

\begin{prop}
$(A,\Delta)$ and $V$ as above. Fix $\lambda : A \otimes V \to A$ and
$\kappa = \kappa_A \oplus \kappa_V : V \wedge V \to A \oplus V$ as above.
\begin{enumerate}
\item Then $\hb$ is a cocommutative algebra with (the image of) $V$
primitive, if
\begin{equation}\label{Ehopf1}
\Delta(\lambda(a,v)) = \sum \lambda(\one{a},v) \otimes \two{a} + \sum
\one{a} \otimes \lambda(\two{a}, v), \qquad
\kappa_A(v,v') \text{ is primitive},
\end{equation}
for all $v,v' \in V,\ a \in A$.
The converse is true if $\hb$ has the PBW property.

\item Suppose $A$ is a cocommutative bialgebra (with counit $\vi$). Then
$\hb$ is a cocommutative bialgebra with $V$ primitive, if \eqref{Ehopf1}
holds and $\im \lambda \subset \ker \vi$. The converse is true if $\hb$
has the PBW property.

\item Suppose $A$ is a cocommutative Hopf algebra (with counit $\vi$ and
antipode $S$). Then $\hb$ is a cocommutative Hopf algebra with $V$
primitive, if \eqref{Ehopf1} holds and moreover,
\[
\im \lambda \subset \ker \vi, \qquad
S(\lambda(a,v)) = \sum \lambda( S(\one{a}), \two{a}(v) ).
\]
The converse is true if $\hb$ has the PBW property.
\end{enumerate}
In particular, notice that in all three cases, the structure on $A$
automatically extends to $\scrh_{0,0} = \Sym(V) \rtimes A$, and more
generally, to all $\scrh_{0,\kappa}$ for which $\im \kappa_A$ is
primitive.
\end{prop}

\begin{proof}
To prove the first part, suppose $\hb$ has the PBW property. If $V$ is
primitive, then we compute in the algebra $\hb \otimes \hb$:
\begin{align*}
\Delta(\lambda(a,v)) = &\ \Delta(av) - \sum \Delta(\one{a}(v) \two{a})\\
= &\ \Delta(a) \Delta(v) - \sum \Delta(\one{a}(v)) \Delta(\two{a})\\
= &\ \sum \lambda(\one{a},v) \otimes \two{a} + \sum \one{a} \otimes
\lambda(\two{a}, v),
\end{align*}

\noindent and similarly,
\begin{align*}
&\ \Delta(\kappa_A(v,v')) -
( 1 \otimes \kappa_A(v,v') + \kappa_A(v,v') \otimes 1)\\
= &\ \Delta(\kappa_A(v,v')) + \Delta(\kappa_V(v,v')) -
( 1 \otimes \kappa(v,v') + \kappa(v,v') \otimes 1)\\
= &\ \Delta( [v,v'] )  - (1 \otimes [v,v'] + [v,v'] \otimes 1) = 0.
\end{align*}

\noindent Since $\hb$ has the PBW property, the above equalities in fact
hold inside $V \otimes A$ and $A \otimes A$, which inject into $\hb
\otimes \hb$ by Theorem \ref{Tpbw}. To prove the converse, even when
$\hb$ need not have the PBW property, one uses essentially the same
computations as above (but slightly rearranged).

This proves the first part. For the second part, that $\vi(\im \kappa_A)
= 0$ follows from its primitivity, and that $\vi(\im \lambda) = 0$
follows from applying $\vi$ to the defining relations. The third part now
follows from the following computation (using that $S|_V = -\id_V$ as $V$
is primitive):
\begin{align*}
S(\lambda(a,v)) = &\ S(a) S(v) - \sum S(\one{a}(v)) S(\two{a})\\
= &\ (-v) S(a) + \sum S(\two{a})(\one{a}(v)) S(\three{a}) + \sum
\lambda(S(\two{a}), \one{a}(v)),
\end{align*}
\noindent and now applying the cocommutativity of $A$, to cancel the
first two expressions.
\end{proof}

%{{{1 Section 4.1 - Symplectic reflections in bialgebras
\subsection{Symplectic reflections in bialgebras}

Our next goal is to show that the notion of ``symplectic reflections''
generalizes to arbitrary cocommutative bialgebras. The following result
extends to such a setting, its group-theoretic counterparts in
\cites{Dr,EG}.

\begin{prop}\label{Prefl}
Suppose $\bk$ is a field, and $(A, \Delta, \vi)$ is a cocommutative
$\bk$-bialgebra. Suppose $\kappa_V = 0$ and $\hb$ has the PBW property.
Given $0 \neq a' \in A$, suppose there exists nonzero $a'' \in A$ and a
vector space complement $U$ to $\bk a''$ in $A$ such that
\[
\Delta(\im \kappa_A) \subset \bk(a' \otimes a'') \oplus (A \otimes U),
\]

\noindent but $\Delta(\im \kappa_A) \nsubseteq A \otimes U$. Then $a' -
\vi(a') \in \End_\bk V$ has image with dimension at most $2$.
\end{prop}

\noindent In other words, if $\kappa_A$ is supported on $a' \otimes a''$,
then $a' - \vi(a')$ is akin to a symplectic reflection \cite{EG}.
For instance, for symplectic reflection algebras as in \cites{Dr,EG},
with $A = \bk W$ a group ring, if $a' = g \in W$, then choose $U :=
\sum_{g' \neq g} \bk g'$.

\begin{proof}
We may assume throughout that $a' \neq \vi(a')$. By choice of $a'$, there
exist $x,y \in V$ such that $\Delta (\kappa_A(x,y)) - r (a' \otimes a'')
\in A \otimes U$, for some $r \in \bk^\times$. We now claim that for all
$v \in V$,
\[
(a' - \vi(a'))(v) \in \bk v_x + \bk v_y, \quad \text{where} \quad
v_x := (a' - \vi(a'))(x),\ v_y := (a' - \vi(a'))(y).
\]

\noindent To show the claim, consider the Jacobi identity
\eqref{Ejacobi2} for $v_1 = x, v_2 = y, v_3 = v$, which yields:
\[
\sum_\circlearrowright \left( \one{\kappa_A(v_1, v_2)} -
\vi(\one{\kappa_A(v_1, v_2)}) \right) (v_3) \two{\kappa_A(v_1, v_2)} =
0.
\]

\noindent Denote the summand by $f(x,y,v)$. Now split the term
$\kappa_A(x,y)$ (and the other two cyclically permuted such terms) into
their $a' \otimes a''$-components and $A \otimes U$-components.
Hence there exist $r_{xy} = r, r_{yv}, r_{vx} \in \bk$ such that by the
PBW property,
\[
r (a' - \vi(a'))(v) \otimes a'' + r_{yv} (a' - \vi(a'))(x) \otimes a'' +
r_{vx} (a' - \vi(a'))(y) \otimes a'' \in V \otimes U.
\]

\noindent This shows that the left-hand side vanishes. The claim now
follows by the PBW property.
\end{proof}
%}}}

%{{{1 Section 4.2 - Yetter--Drinfeld condition
\subsection{Yetter--Drinfeld condition}

\noindent In the remainder of this section, we work with Hopf algebras.
\textbf{Assume} throughout this subsection that $A$ is a $\bk$-free
cocommutative $\bk$-Hopf algebra, and $V$ is a $\bk$-free $A$-module.
In this case it is easy to verify that the $A$-action on $TV$
(respectively, $\Sym(V)$) agrees with the adjoint action of $A$:
$\ad a(x) := \sum \one{a} x S(\two{a}) = a(x)$, for $x \in TV$
(respectively, $\Sym(V)$).

Our goal is to show that one of the conditions in Theorem \ref{Tpbw}
required for the PBW property to hold is equivalent to a compatibility
condition called the \textit{Yetter--Drinfeld condition} (see e.g.
\cite{BaBe}*{Theorem 3.3}). %[Lemma 2.12.3, Proposition 4.3.3]
To state the result, we require some preliminaries.

\begin{prop}\label{Preln}
Suppose a $\bk$-Hopf algebra $A$ acts on a free $\bk$-module $V$, and a
$\bk$-algebra $B$ contains $A,V$.
\begin{enumerate}
\item Then the following relations in $B$ are equivalent for all $v \in
V$:
\begin{enumerate}
\item $\sum \one{a} v S(\two{a}) = a(v)$ for all $a \in A$.
\item $a v = \sum \one{a}(v) \two{a}$ for all $a \in A$.\medskip

\noindent If $A$ is cocommutative, then both of these are also equivalent
to:

\item $va = \sum \one{a} S(\two{a})(v)$ for all $a \in A$.
\end{enumerate}\medskip

\noindent Now suppose in the remaining parts that the conditions (a),(b)
hold.\medskip

\item Suppose $A$ is cocommutative. Then $\tau : A \otimes V \to V
\otimes A$, given by $a \otimes v \mapsto \sum \one{a}(v) \otimes
\two{a}$, as well as $\tau^{op} : V \otimes A \to A \otimes V$, given by
$v \otimes a \mapsto \sum \one{a} \otimes S(\two{a})(v)$, are $A$-module
isomorphisms that are inverse to one another.

\item Any unital subalgebra $M$ of $B$ that is also an $A$-submodule (via
$\ad$), is an $A$-(Hopf) module algebra under the action
\[
a(m) := \ad a(m) = \sum \one{a} m S(\two{a})\ \forall a \in A,\ m \in M.
\]
\end{enumerate}
\end{prop}

\noindent The proof of the following result is standard and is hence
omitted. The result may be applied to $B = \hb$. Note as in
\cite{SW2}*{\S 4} that the map $\tau$ is an isomorphism of the
Yetter--Drinfeld modules $A \otimes V$ and $V \otimes A$, called the
``braiding''.

The following preliminary result can (essentially) be found in
\cite{Jo}*{Lemma 1.3.3}. To state the result, recall that given a module
$M$ over a Hopf $\bk$-algebra $A$, the \textit{$\vi$-weight space}
$M_\vi$ is $\{ m \in M : a \cdot m = \vi(a) m\ \forall a \in A \}$.

\begin{lemma}\label{LJo}
Given a Hopf algebra $A$ and a $\bk$-algebra map $\varphi : A \to B$, the
centralizer of $\varphi(A)$ in $B$ is the weight space $B_\vi$ (where $B$
is an $A$-module via: $a \cdot b := \sum \varphi(\one{a}) b
\varphi(S(\two{a}))$).
\end{lemma}

\noindent Consequently, the deformation $\scrh_{0,\kappa}$ is
commutative if and only if $A = A_\vi$ under the adjoint action
(equivalently, $A$ is commutative), $V = V_\vi$ (under the given
$A$-action), and $\kappa \equiv 0$.\medskip

We now discuss the Yetter--Drinfeld condition in detail. In the following
result, $\tau^{op} : M \otimes A \to A \otimes M$ is defined as in
Proposition \ref{Preln}(2), and $A^{ad}, A^{mult}$ refer to different
$A$-module structures on $A$ (via the adjoint action, and via left
multiplication respectively).

\begin{prop}\label{Pyd}
Suppose $A$ is a Hopf $\bk$-algebra, $V,M$ are $\bk$-free $A$-modules,
and $\kappa \in \Hom_\bk(V \wedge V, M)$. Suppose $(B, \mu_B, 1_B)$ is an
(associative) $\bk$-algebra containing $A,M$, with the additional
relations $m \cdot a = \mu_B(\tau^{op}(m \otimes a))$ in $B$. The
following are equivalent in $B$:
\begin{enumerate}
\item $\kappa : V \wedge V \to M$ is $A$-{\em equivariant}, or an
$A$-module map:
\[
a(\kappa(v,v')) = \sum \kappa(\one{a}(v), \two{a}(v'))\ \forall a \in A,
v,v' \in V.
\]

\item $\kappa$ satisfies the {\em Yetter--Drinfeld (compatibility)
condition}, i.e.
\[
\tau^{op} \left( \sum \kappa(\one{a}(v), v') \two{a} \right) = \sum
\one{a} \kappa(v, S(\two{a})(v'))\ \forall a \in A, v,v' \in V.
\]

\item $\kappa$ is $A$-{\em compatible}: $\displaystyle a \kappa(v,v') =
\sum \kappa(\one{a}(v), \two{a}(v')) \three{a}\ \forall a,v,v'$.

\item $\kappa$ satisfies: $\displaystyle \kappa(v,v') a = \sum \one{a}
\kappa(S(\two{a})(v), S(\three{a})(v'))\ \forall a,v,v'$.\medskip

\noindent If $\kappa$ also satisfies: $\kappa(a(v),v') =
\kappa(v,S(a)(v'))$ for all $v,v',a$, then these are also equivalent
to:\medskip

\item $\im \kappa$ commutes (in $B$) with all of $A$.
\end{enumerate}
\end{prop}

\noindent The proof is a relatively straightforward exercise in
computations involving Hopf algebras, and is hence omitted. We remark
that the proof uses Proposition \ref{Preln}, Lemma \ref{LJo} and that $A$
is cocommutative.

To conclude this section, we point out how the Yetter--Drinfeld condition
arises, as in \cite{BaBe}*{Theorem 3.3}: in the associative algebra $B$
above, compute $v' \cdot a \cdot v$ in two different ways (i.e. using the
maps $\tau, \tau^{op}, \kappa$). Then,
\[
\sum \one{a} \kappa(v, S(\two{a})(v')) = \sum \one{a}(v) \two{a}
S(\three{a})(v') - v' a v = \sum \kappa(\one{a}(v), v') \two{a},
\]

\noindent and this is precisely the Yetter--Drinfeld condition.
%}}}

\section{Generalized nil-Coxeter algebras and grouplike algebras}\label{S5}

In the remainder of this paper, we introduce a class of cocommutative
algebras that incorporates group algebras as well as nil-Coxeter algebras
and their generalizations, which are necessarily not bialgebras or Hopf
algebras. We then study the Jacobi identity \eqref{Ejacobi2} in detail;
this is useful in classifying PBW deformations over nil-Coxeter algebras.

We begin by setting notation concerning unitary/complex reflection
groups.

\begin{definition}\label{Dbraid}
A \textit{Coxeter matrix} is a symmetric matrix $A := (a_{ij})_{i,j \in
I}$ indexed by a finite set $I$ and with integer entries, such that
$a_{ii} = 1$ and $2 \leqslant a_{ij} \leqslant \infty$ for all $i \neq
j$. Given a Coxeter matrix $A$, define the corresponding \textit{braid
group} $\mathcal{B}_W = \mathcal{B}_{W(A)}$ to be the group generated by
\textit{simple reflections} $\{ s_i : i \in I \}$, satisfying the
\textit{braid relations} $s_i s_j s_i \cdots = s_j s_i s_j \cdots$ for
all $i \neq j$, with precisely $a_{ij}$ factors on either side. Finally,
define the \textit{Coxeter group} $W = W(A)$ to be the quotient of the
braid group by the additional relations $s_i^2 = 1\ \forall i$. More
broadly, given an integer tuple ${\bf d}$ with $d_i \geqslant 2\ \forall
i \in I$, define the corresponding \textit{generalized Coxeter group}
$W({\bf d})$ to be the quotient of $\mathcal{B}_{W(A)}$ by $s_i^{d_i} =
1\ \forall i$.
\end{definition}

\noindent We now introduce the corresponding families of generalized
(nil-)Coxeter groups and algebras. This involves considering the
``non-negative part'' of the braid group, i.e., the Artin monoid.

\begin{definition}\label{Dnilcox}
Given  a Coxeter matrix $A$, first define the \textit{Artin monoid}
$\mathcal{B}_{W_A}^{\geqslant 0}$ to be the monoid generated by $\{ T_i :
i \in I \}$ modulo the braid relations. Now given an integer vector ${\bf
d} = (d_i)_{i \in I}$ with each $d_i \geqslant 2$, define the
\textit{generalized nil-Coxeter algebra} $NC_{W_A}({\bf d})$ as:
\begin{equation}
NC_{W_A}({\bf d}) := \frac{\bk \tangle{T_i, i \in I}}
{(\underbrace{T_i T_j T_i \cdots}_{a_{ij}\ times} =
\underbrace{T_j T_i T_j \cdots}_{a_{ij}\ times}, \
T_i^{d_i} = 0, \ \forall i \neq j \in I)}
= \frac{\bk \mathcal{B}_{W_A}^{\geqslant 0}}{(T_i^{d_i} = 0\ \forall i)}.
\end{equation}
\end{definition}

\begin{remark}\label{Rweak}
The algebras $NC_W({\bf d})$ provide a large family of examples of
cocommutative algebras via $\Delta(T_i) := T_i \otimes T_i$ for all $i
\in I$ (and extending $\Delta$ by multiplicativity).
Moreover, \textbf{no} algebra $NC_W({\bf d})$ can be a (weak) bialgebra
under this coproduct. This is because any counit $\vi$ necessarily maps
the nilpotent element $T_i$ to $0$; but $T_i$ is grouplike so $\vi(T_i) =
1$.
\end{remark}

Generalized nil-Coxeter algebras $NC_W({\bf d})$ include the well-studied
case (see the Introduction) of the nil-Coxeter algebra $NC_W$, where $d_i
= 2\ \forall i$. Note that $\dim NC_W({\bf d}) \geqslant NC_W$, as
$NC_W({\bf d})$ surjects onto $NC_W$.
Moreover, if $W$ is finite, then $\dim NC_W((2, \dots, 2)) = |W| <
\infty$; see e.g. \cite{Hum}*{Chapter 7}. Notice that there are other
finite-dimensional algebras of the form $NC_W({\bf d})$. For instance,
$NC_{A_1}(d) \cong \bk[T_1] / (T_1^d)$ is finite-dimensional; hence, so
is the algebra $NC_{A_1^n}((d_1, \dots, d_n))$ with all $d_i \geqslant
2$.
This question is completely resolved in related work \cite{Khnilcox},
where we characterize the generalized nil-Coxeter algebras $NC_W({\bf
d})$ that are finite-dimensional. We show that apart from the usual
nil-Coxeter algebras $NC_W((2,\dots,2))$, there is precisely one other
family of type-$A$ algebras, $NC_A((2,\dots,2,d))$ with $d > 2$, which
are finite-dimensional. See \cite{Khnilcox}*{Theorems A,C} for further
details.

%{{{1 Section 5.1 - Grouplike algebras
\subsection{Grouplike algebras}

We begin by unifying the group algebras $\bk W$ and the algebras
$NC_W({\bf d})$ (as well as other algebras considered in the literature)
in the following way.

\begin{definition}
A \textit{grouplike algebra} is a unital $\bk$-algebra $A$, together with
a distinguished $\bk$-basis $\{ T_m : m \in M_A \}$ containing the unit
$1_A$, such that the map $\Delta : A \otimes A, \ T_m \mapsto T_m \otimes
T_m$ is an algebra map.
\end{definition}

\begin{remark}\label{Rgrouplike}
Observe from the definitions that the grouplike elements $g := \sum_{m
\in M_A} c_m T_m$ in a grouplike algebra $A$ can all be easily
identified. Indeed, if $g \neq 0$ and $\bk$ is a domain, then
\[
\sum_{m,m' \in M_A} c_m c_{m'} T_m \otimes T_{m'} = \Delta(g) = \sum_{m
\in M_A} c_m T_m \otimes T_m,
\]

\noindent from which it follows that the sum is a singleton, with
coefficient $1$. Thus $g = T_m$ for some $m$. As a consequence, it
follows that the set $\{ T_m : m \in M_A \} \sqcup \{ 0 \}$ is closed
under multiplication, making it a monoid with both a unit and a zero
element. This is formalized presently.
\end{remark}

Notice that every grouplike algebra is a cocommutative algebra with
coproduct. (Henceforth we will suppress the monoid operation $*$ when it
is clear from context.) As we presently show, generalized Coxeter groups
and generalized nil-Coxeter algebras are examples of grouplike algebras.
First we introduce the following notation.

\begin{definition}
We work over a unital commutative ring $\bk$.
\begin{enumerate}
\item Given a monoid $(M,*)$, its \textit{monoid algebra}, denoted by
$\bk M$ and analogous to the notion of a group algebra, is a
$\bk$-algebra that has $\bk$-basis $M$, with the multiplication in $M$
extended by linearity to all of $\bk M$.

\item A \textit{zero/absorbing/annihilating element} in a monoid $M$ is
an element $0_M \in M$ such that $0_M * m = m * 0_M = 0_M$ for all $m \in
M$. Such an element is necessarily unique in $M$ (and idempotent).
\end{enumerate}
\end{definition}

We now present several examples of (cocommutative) grouplike algebras.
\begin{enumerate}
\item Every monoid algebra $\bk M$ is a grouplike algebra, using $T_m :=
m$ for all $m$. This includes the group algebra of every (generalized)
Coxeter group.

\item Suppose $M$ contains a zero element $0_M$. Then $\bk 0_M$ is a
two-sided ideal in the monoid algebra $\bk M$, and so $\bk M / \bk 0_M$
is also a grouplike algebra with basis $\{ T_m : m \in M \setminus 0_M
\}$. The previous example is a special case, since to each monoid $M$ we
can formally attach a zero element $0$, to create a new monoid with zero
element $0$.

\item Another special case of the preceding example is a nil-Coxeter
algebra $NC_W$. This corresponds to the monoid $W \sqcup \{ 0_W \}$, with
$T_w * T_{w'} := 0_W$ if $\ell(ww') > \ell(w) + \ell(w')$ in $W$. More
generally, define for $k \in \mathbb{N}$ the ideal $\mathcal{I}_k$ to be
the $\bk$-span of $\{ T_w : \ell(w) \geqslant k \}$. Then $NC_W /
\mathcal{I}_k$ is a grouplike algebra, with distinguished basis $\{
\overline{T_w} : \ell(w) < k \}$.

\item The generalized nil-Coxeter algebra $NC_{A_1^n}((d_1, \dots,
d_n))$, with $d_i \geqslant 2$ for all $i$, is yet another example of the
above construction. In this case we use the monoid
\[
M := \{ 0 \} \sqcup \times_i \{ 1, \dots, d_i-1 \},
\]
with $(e_i)_i * (e'_i)_i$ equal to $(e_i + e'_i)_i$ if $\max_i (e_i +
e'_i - d_i) < 0$, and $0$ otherwise.

\item As a final example, recall the \textit{$0$-Hecke algebra}
\begin{equation}
\mathcal{H}_W(0) := \frac{\bk \mathcal{B}_W^{\geqslant 0}}{(T_i^2 = T_i\
\forall i \in I)},
\end{equation}

\noindent where $\mathcal{B}_W^{\geqslant 0}$ is as in Definition
\ref{Dnilcox}. This algebra was defined in \cite{Nor} and has been
extensively studied since; see \cites{Fa,He,TvW} and the references
therein. We recall from \cite{HST} that $\mathcal{H}_W(0)$ is the monoid
algebra of a monoid in bijection with $W$. As we presently show, it is
also a grouplike algebra with distinguished basis $\{ T_w : w \in W \}$.
\end{enumerate}

Given the profusion of Coxeter-theoretic examples above, it is desirable
to consider a subclass of grouplike algebras that incorporates them all
in a systematic manner. We now present such a family.

\begin{definition}
Given a Coxeter matrix $A$ and an integer vector ${\bf d}$ with $2
\leqslant d_i \leqslant \infty\ \forall i$, a \textit{generic Hecke
algebra} is any algebra of the form
\begin{equation}\label{EgenHecke}
\mathcal{E}_W({\bf d}, {\bf p}) := \frac{\bk
\mathcal{B}_W^{\geqslant 0}}{(T_i^{d_i} = p_i(T_i)\ \forall i \in I)},
\end{equation}
where $W = W_A$, and $p_i \in \bk[T]$ has degree at most $d_i - 1$ for $i
\in I$.
\end{definition}

These algebras are so named after the family of ``generic Hecke
algebras'' studied in \cites{BMR1,BMR2}; however, unlike
\textit{loc.~cit.}, we do not require the $p_i$ to be equal when the
corresponding simple reflections are conjugate in $W$.
Note that all generalized (nil-)Coxeter groups and algebras as in
Definition \ref{Dnilcox} are covered by our definition.

Recall that our goal in the present paper is to study cocommutative
algebras. Thus, we now study when generic Hecke algebras provide
examples of such algebras.

\begin{prop}
Suppose $\bk$ is a domain, $W = W_A$ is a Coxeter group, and ${\bf d},
{\bf p}$ are as in Equation \eqref{EgenHecke}.
\begin{enumerate}
\item The map $\Delta : T_i \mapsto T_i \otimes T_i$ extends to make
$\mathcal{E}_W({\bf d}, {\bf p})$ a (cocommutative) grouplike algebra, if
for all $i \in I$, $p_i(T)$ is either zero or equals $T^{e_i}$ for some
$0 \leqslant e_i < d_i$.

\item $\mathcal{E}_W({\bf d}, {\bf p})$ is a bialgebra if for all $i \in
I$, $p_i(T) = T^{e_i}$ for some $0 \leqslant e_i < d_i$.

\item $\mathcal{E}_W({\bf d}, {\bf p})$ is a Hopf algebra if $p_i(T) = 1\
\forall i \in I$.
\end{enumerate}
The converse statements are all true if for all $i$, the vectors $1, T_i,
\dots, T_i^{d_i - 1}$ are $\bk$-linearly independent in
$\mathcal{E}_W({\bf d}, {\bf p})$.
\end{prop}

Notice that the last condition is not always true. For instance, standard
arguments as in \cite{Ko}*{Introduction} show that the condition fails to
hold in a generalized Coxeter group $W$ (or $\bk W$ to be precise)
whenever $a_{ij}$ is odd, $p_i = 1$ is constant for all $i$, and $d_i
\neq d_j$. However, the condition does hold in group algebras, $0$-Hecke
algebras, and nil-Coxeter algebras corresponding to Coxeter groups.

\begin{proof}
We begin by showing the first three assertions. Suppose for all $i$ that
$p_i(T) = 0$ or $T^{e_i}$ for some $0 \leqslant e_i < d_i$. Then it is
easily verified that $\Delta : T_i \mapsto T_i \otimes T_i$ extends to
the tensor algebra over the $T_i$, hence to the Artin monoid $\bk
\mathcal{B}_W^{\geqslant 0}$, and hence to  $\mathcal{E}_W({\bf d}, {\bf
p})$. Similarly one verifies that a counit that sends $T_i$ to $1$ for
all $i$, can be extended to $\mathcal{E}_W({\bf d}, {\bf p})$ if $p_i(T)
= T^{e_i}$ for all $i$. Finally, an antipode that sends $T_i$ to
$T_i^{-1} = T_i^{d_i - 1}$ can be extended to $\mathcal{E}_W({\bf d},
{\bf p})$.

The ``converse'' statements are slightly harder to show. Suppose $1, T_i,
\dots, T_i^{d_i - 1}$ are $\bk$-linearly independent in
$\mathcal{E}_W({\bf d}, {\bf p})$.
To show (the converse of) (1), notice that every algebra of the form
$\mathcal{E}_W({\bf d}, {\bf p})$ is a quotient of $\bk
\mathcal{B}_W^{\geqslant 0}$, so it suffices to classify the polynomials
$p_i$ such that the ideal generated by all $T_i^{d_i} - p_i(T_i)$ is a
coideal. Define $p_i(T) := \sum_{j=0}^{d_i-1} p_{ij} T^j$, and compute
using the multiplicativity of $\Delta$:
\begin{equation}\label{Ecompute}
\Delta(T_i^{d_i}) = T_i^{d_i} \otimes T_i^{d_i} = \sum_{j,k=0}^{d_i -
1} p_{ij} p_{ik} T_i^j \otimes T_i^k, \qquad
\sum_{j=0}^{d_i-1} \Delta(p_{ij} T_i^j) = \sum_{j=0}^{d_i-1} p_{ij} T_i^j
\otimes T_i^j.
\end{equation}

It follows by the assumptions that each nonzero $p_i(T)$ is a monomial
$p_{ij} T^j$, with $p_{ij}^2 = p_{ij}$ in the domain $\bk$. This proves
(1).
To show (2), it suffices to produce a counit $\vi$ that is compatible
with the coproduct. Since $T_i$ is grouplike, it follows that $\vi(T_i)$
must equal $1$ for all $i$. This is indeed compatible with the relations
$T_i^{d_i} = T_i^{e_i}$, which shows one implication. On the other hand,
the relation $T_i^{d_i} = 0$ implies $\vi(T_i) = 0$, a contradiction.

Finally, we show (3). If $p_i(T) = 1$ for all $i$ then
$\mathcal{E}_W({\bf d}, {\bf p})$ is a group algebra, hence a Hopf
algebra. Conversely, suppose $p_i(T) = T^{e_i}$ for some $0 < e_i < d_i$
and $i \in I$. Then from above, the subalgebra generated by $T_i$ is
isomorphic to $\bk[T] / (T^{d_i} - T^{e_i})$, which surjects onto the
algebra $\bk[T] / (T^2 - T)$. This is precisely the $0$-Hecke algebra of
type $A_1$, in which one knows that $T$ is not invertible, yet $T$ is
grouplike. Thus $T_i$ is not invertible in $\mathcal{E}_W({\bf d}, {\bf
p})$.
\end{proof}

\begin{remark}
Let $A := \mathcal{E}_W({\bf d}, {\bf p})$. If $p_i(T) = 0\ \forall i$,
and $M := {\rm span}_\bk \{ T_i : i \in I \}$, then $A M = M A = A M A =:
\m$ is a maximal ideal of $A$. This is because $\m$ is a quotient of the
tensor algebra $T_\bk M$, by relations that strictly lie in the
augmentation ideal $T^+_\bk M$.
\end{remark}
%}}}

%{{{1 Section 5.2 - The Jacobi identity for grouplike algebras
\subsection{The Jacobi identity for grouplike algebras}

Having defined grouplike algebras and presented examples of them, we
specialize the conditions in the PBW Theorem \ref{Tpbw} to such a
setting. For instance, if $\lambda, \kappa_V$ are identically zero, and
$A$ is a group algebra $\bk G$ as in \cites{Dr,EG}, then defining
$\kappa_A(v,v') := \sum_{g \in G} \kappa_g(v,v') T_g$, we see easily that
the $A$-compatibility of $\kappa_A$ is equivalent to the following
condition found in \textit{loc.~cit.}:
\[
\kappa_{g h g^{-1}}(T_g(v), T_g(v')) = \kappa_g(v,v'), \qquad \forall g,h
\in G,\ v,v' \in V.
\]

Our goal in the remainder of this section is to study the Jacobi
identity \eqref{Ejacobi2} in the case $\kappa_V \equiv 0$, over a
grouplike algebra $A$.

\begin{stand}
For the remainder of this section, $\bk$ is a field and $\kappa_V \equiv
0$.
\end{stand}

We begin by setting notation.
Define the fixed point space of $a \in A$ and its codimension:
\begin{equation}
\fix(a) := \{ v \in V : a(v) = v \}, \qquad
d_a := \codim_V \fix(a).
\end{equation}

\noindent Thus, $d_a = \dim_\bk \im(\id_V - a)$.

Now suppose we have fixed a $\bk$-basis $\{ a_j : j \in J_1 \}$ of $A$.
Then we will write
\begin{equation}\label{Ebasis}
\kappa(x,y) = \kappa_A(x,y) =: \sum_{j \in J_1} \kappa_j(x,y) a_j, \qquad
\forall x,y \in V.
\end{equation}

\noindent Thus, $\kappa_j$ is a skew-symmetric bilinear form on $V$. We
also define $\rad(\kappa_j)$ to be the \textit{radical} of the bilinear
form, $\rad(\kappa_j) := \{ v \in V : \kappa_j(v, V) \equiv 0
\}$.
Specifically, this notation will be applied to a grouplike algebra $A$
with a distinguished basis $\{ T_m : m \in M_A \}$ of grouplike elements;
see Remark \ref{Rgrouplike}. In this setting, we will write $\kappa_{T_m}
= \kappa_m$ and $d_{T_m} = d_m$.\medskip

We now characterize the Jacobi identity in this general setting.

\begin{theorem}\label{Tgeneric}
Suppose $\kappa_V \equiv 0$.
\begin{enumerate}
\item Suppose $A$ contains a grouplike element $T_m$ and a vector space
complement $V_0$ to $\bk T_m$, such that $\Delta(V_0) \subset V_0 \otimes
V_0$. Extend $T_m$ to any basis of $V_0$. Now if the Jacobi identity
\eqref{Ejacobi2} holds in $\hb$ (with $\kappa_V \equiv 0$), then one of
the following conditions holds:
\begin{enumerate}
\item $\kappa_m \equiv 0$.
\item $T_m \equiv \id_V$, i.e. $d_m=0$.
\item $d_m$ is 1 or 2, and $\rad(\kappa_m)$ is a subspace of $\fix(T_m)$,
of codimension $2 - d_m$.
\end{enumerate}

\item Conversely, if $A$ is a grouplike algebra with distinguished
$\bk$-basis $\{ T_m : m \in M_A \}$ of grouplike elements, and for each
$m \in M_A$ one of the above three conditions holds, then the Jacobi
identity \eqref{Ejacobi2} holds in $\hb$ (with $\kappa_V \equiv 0$).
\end{enumerate}
\end{theorem}

For completeness, we remark that part (1) extends to arbitrary grouplike
algebras a result found in \cites{Dr,EG} for $A$ a group algebra; see also
\cites{Gr,SW3}.

\begin{proof}
Write out the Jacobi identity \eqref{Ejacobi2} using the distinguished
$\bk$-basis of $A$, and isolate the $T_m$-component to get:
\[
\sum_\circlearrowright v_1 \kappa_m(v_2, v_3) = \sum_\circlearrowright
\kappa_m(v_2, v_3) T_m(v_1),
\]

\noindent or equivalently, for all $x,y,z \in V$,
\begin{equation}\label{Ejac}
\kappa_m(y,x) (\id_V - T_m)(z) = \kappa_m(y,z) (\id_V - T_m)(x) +
\kappa_m(z,x) (\id_V - T_m)(y).
\end{equation}

\noindent Before proving the two parts, we make two observations. First,
it follows from \eqref{Ejac} that $\kappa_m \equiv 0$ or $\rad(\kappa_m)
\subset \fix(T_m)$. Moreover, if $\rad(\kappa_m) \subset \fix(T_m)$ has
codimension at most $1$, then by the skew-symmetry of $\kappa_m$ it is
clear that $\fix(T_m)$ is $\kappa_m$-isotropic.

\begin{enumerate}
\item Suppose the Jacobi identity holds. Assume $\kappa_m$ is not
identically zero; thus, choose $x,y$ so that $\kappa_m(y,x) \neq 0$. Then
Equation \eqref{Ejac} implies that $\im (\id_V - T_m) \subset \bk x' +
\bk y'$, where $x' := (\id_V - T_m)(x)$ and $y' := (\id_V - T_m)(y)$.
(This is similar to the proof of Proposition \ref{Prefl}.)
In particular, $d_m = \dim_\bk \im(\id_V - T_m) \leqslant 2$ if $\kappa_m
\not\equiv 0$.

If $d_m = 0$ then assertion (b) holds, so we may assume now that $d_m$ is
$1$ or $2$. Also notice by Equation \eqref{Ejac} that $\rad(\kappa_m)
\subset \fix(T_m)$, so it remains to show that the codimension is $2 -
d_m$.

First suppose $d_m = 2$, whence $x',y'$ are linearly independent. We
claim that $\rad(\kappa_m) \supset \fix(T_m)$. Indeed, suppose $z \in
\fix(T_m)$. Then Equation \eqref{Ejac} yields:
\begin{equation}\label{Ekappa}
\kappa_m(y,z) x' + \kappa_m(z,x) y' = 0.
\end{equation}

\noindent Similarly, replacing $x$ by $z' \in \ker(\id_V - T_m)$ yields:
$\kappa_m(z,z') y' = 0$. From this and \eqref{Ekappa}, it follows that
$\kappa_m(z,-)$ kills $x,y$ as well as $\ker(\id_V - T_m) = \fix(T_m)$.
Hence it kills their $\bk$-span, which is all of $V$.

The final case is when $d_m = 1$. Fix $v_1 \not\in \fix(T_m)$; thus $V =
\bk v_1 \oplus \fix(T_m)$. We may assume $v_1 \not\in \rad(\kappa_m)$.
Indeed, if instead $\kappa_m(v_1, V) = 0$, then $\kappa_m(v'_0, v_0) \neq
0$ for some $v_0, v'_0 \in \fix(T_m)$, since $\kappa_m \not\equiv 0$.
Then $\kappa_m(v_1 + v'_0, v_0) \neq 0$, so we can replace $v_1$ by $v_1
+ v'_0$.
Proceeding, notice that $\kappa_m(v_1, v_0) \neq 0$ for some $v_0 \in
\fix(T_m)$. Now define $V_0 := \{ v \in \fix(T_m) : \kappa_m(v_1, v) = 0
\}$; then $\fix(T_m) = \bk v_0 \oplus V_0$, and $V_0 \supset
\rad(\kappa_m)$ from the observations following \eqref{Ejac}. Finally,
applying \eqref{Ejac} to $z, y \in \fix(T_m),\ x = v_1$ shows that
$\fix(T_m)$ is $\kappa_m$-isotropic. Hence $V_0 = \rad(\kappa_m)$.

\item Conversely, suppose $A$ is grouplike with basis $\{ T_m : m \in M_A
\}$ as given. We are to show that Equation \eqref{Ejac} holds for all $m
\in M_A$. Certainly this holds if $\kappa_m \equiv 0$ or $T_m \equiv
\id_V$. Thus we assume henceforth that $\kappa_m \not\equiv 0$, and show
Equation \eqref{Ejac} for a fixed $m \in M_A$, in the two cases $d_m =
1,2$. First suppose $d_m = 2$, and $x,y \in V$ are linearly independent
modulo $\rad(\kappa_m)$. Notice that $\kappa_m(v,v')$ is nonzero only if
$v,v'$ are independent modulo $\rad(\kappa_m)$, so it suffices to prove
\eqref{Ejac} with $x,y$ as above, whence $z = \alpha x + \beta y + v$ for
some $\alpha,\beta \in \bk$ and $v \in \rad(\kappa_m) = \fix(T_m)$. In
this case it is easily shown that both sides of \eqref{Ejac} equal
$\kappa_m(y,x) \cdot (\id_V - T_m)(\alpha x + \beta y)$.

Finally, suppose $d_m = 1$, with $V \supset \fix(T_m) \supset
\rad(\kappa_m)$ a chain of codimension one subspaces. Choose $x \in V
\setminus \fix(T_m)$ and $y \in \fix(T_m) \setminus \rad(\kappa_m)$; once
again, if $\kappa_m(v,v')$ is nonzero we may replace $v,v'$ by $x,y$, and
set $z = \alpha x + \beta y + v$ for $v \in \rad(\kappa_m)$. Now both
sides of \eqref{Ejac} are equal to $\kappa_m(y,x) \cdot (\id_V -
T_m)(\alpha x)$.
\end{enumerate}
\end{proof}

Theorem \ref{Tgeneric} is useful in characterizing PBW deformations, via
the following consequence.

\begin{cor}\label{Cniljac}
Suppose $A$ contains a grouplike and nilpotent element $T_m$, and a
vector space complement $V_0$ to $\bk T_m$ such that $\Delta(V_0) \subset
V_0 \otimes V_0$. If the Jacobi identity \eqref{Ejacobi2} holds in $\hb$
with $\kappa_V \equiv 0$, then either $\kappa_m \equiv 0$ or $\dim_\bk V
= 2$.
\end{cor}

\begin{proof}
Since $\id_V - T_m$ is invertible, Theorem \ref{Tgeneric}(1) implies that
either $\kappa_m \equiv 0$, or $d_m = \dim_\bk V$ and $\rad(\kappa_m) =
\fix(T_m) = 0$, whence $d_m = 2$.
\end{proof}

We conclude this section by specializing to the case of a generalized
nil-Coxeter algebra $A = NC_W({\bf d})$. Recall from Remark \ref{Rdim2}
that the condition $\dim_\bk V = 2$ is sufficient for the Jacobi
identities \eqref{Ejacobi1}, \eqref{Ejacobi2} to hold for $\hb$. The
following result shows that over $A = NC_W({\bf d})$ and under the
original setting of $\lambda, \kappa_V \equiv 0$ considered in
\cites{Dr,EG}, either $\kappa_A$ is highly constrained, or else the
condition $\dim_\bk V = 2$ is also necessary.

\begin{theorem}\label{Tnilcox}
Suppose $A = NC_W({\bf d})$ is such that the maximal ideal $\m$ generated
by $\{ T_i : i \in I \}$ is nilpotent. Given an $A$-module $M$, define
$\Prim(M) := \{ m \in M : T_i m = 0\ \forall i \}$.
\begin{enumerate}
\item If $\dim_\bk V \leqslant 2$, then $\scrh_{0, \kappa}$ has the PBW
property if and only if $\im \kappa_V \subset \Prim(V)$ and $\im \kappa_A
\subset \Prim(A^{mult})$.

\item If $\dim_\bk V > 2$, and $\lambda, \kappa_V \equiv 0$, then
$\scrh_{0, \kappa_A}$ has the PBW property if and only if $\kappa_A
\equiv 0$.
\end{enumerate}
\end{theorem}

\noindent Thus (using Remark \ref{Rdim2}), if $\scrh_{0, \kappa_A}$
satisfies the PBW property for $A = NC_W({\bf d})$ finite-dimensional,
then either $\kappa_A \equiv 0$ or $\dim_\bk V = 2$.

We also provide examples of $\Prim(\cdot)$ for generalized nil-Coxeter
algebras. Indeed, $\Prim(A^{mult})$ equals $\bk T_{w_\circ}$ if $A =
NC_W$ is the usual nil-Coxeter algebra over a finite Coxeter group $W$
with unique longest element $w_\circ$. If $A = NC_{A_1^n}((d_1, \dots,
d_n))$, then $\Prim(A^{mult}) = \prod_i T_i^{d_i - 1}$. In both of these
cases, the maximal ideal $\m$ is indeed nilpotent, and hence $A$
satisfies the hypotheses of the above theorem for these families of
generalized nil-Coxeter algebras.

\begin{proof}
Suppose $\m^n = 0 \neq \m^{n-1}$ for some $n \in \mathbb{N}$.
Before proving the result, we consider the following filtration on an
$A$-module $V$:
\begin{equation}\label{Efilt}
V \supset \m V \supset \m^2 V \supset \cdots \supset \m^n V = 0.
\end{equation}

\noindent We fix $k \leqslant n-1$ such that $\m^k V = 0 \neq \m^{k-1}
V$.
\begin{enumerate}
\item By Remark \ref{Rdim2}, and given that $\lambda \equiv 0$, it
suffices to characterize the $A$-compatibilities \eqref{Ecompat1},
\eqref{Ecompat2}, assuming further that $\dim V = 2$. Now observe that
$\m^{k-1} V \subset \Prim(V)$. Choose $v_0 \in \m^{k-1} V$, and $v_1
\not\in \bk v_0$; thus $V = \bk v_0 \oplus \bk v_1$. Now notice that
$\kappa|_{V \wedge V}$ is completely determined by $\kappa(v_0, v_1)$,
since $\dim V = 2$. Thus, we compute using the $A$-compatibility
\eqref{Ecompat1}, for any non-trivial grouplike element $1 \neq T_m \in
NC_W({\bf d})$:
\[
T_m \kappa_A(v_0, v_1) = \kappa_A(T_m(v_0), T_m(v_1)) T_m = 0.
\]

\noindent This equation holds for all non-unital $T_m$, if and only if
$\kappa_m \equiv 0$ for $T_m \not\in \Prim(A^{mult})$.
Similarly, Equation \eqref{Ecompat2} reduces to:
\[
T_m(\kappa_V(v_0, v_1)) = \kappa_V(T_m(v_0), T_m(v_1)) T_m = 0,
\]

\noindent which holds if and only if $\kappa_V(v_0, v_1) \in \Prim(V)$,
as claimed.

\item By Corollary \ref{Cniljac}, we see that $\kappa_A \equiv \kappa_1$,
since each non-unital grouplike element $T_m$ is nilpotent by assumption.
Now as above, Equation \eqref{Ecompat1} reduces to:
\[
T_m \kappa_A(x,y) = \kappa_A(T_m(x), T_m(y)) T_m, \quad \forall m \in
M_A,
\]

\noindent so it follows that $\kappa_A(x,y) = \kappa_A(T_m(x), T_m(y))$
for all non-unital $T_m$ and all $x,y \in V$. Repeated applications of
this fact show that $\kappa_A(x,y) = \kappa_A(T_m^k(x), T_m^k(y)) = 0$.
Conversely, $\scrh_{0,0} = \Sym(V) \rtimes A$ has the PBW property.
\end{enumerate}
\end{proof}

For completeness, we mention two properties of generalized nil-Coxeter
algebras, even though they will not be used in the paper.
First, the algebras $NC_W({\bf d})$, and more generally, every generic
Hecke algebra $\mathcal{E}_W({\bf d}, {\bf p})$, is equipped with an
anti-involution that fixes every generator $T_i$. This is because the
defining relations are preserved by such a map. Such an anti-involution
can be used to construct an exact contravariant duality functor on a
suitable category of $A$-modules, which preserves the simple object $\bk
= A / \m$.

Second, as discussed in \cite{Kho}, for all finite Coxeter groups $W$ the
nil-Coxeter algebra is a Frobenius algebra, by defining a trace map to
kill all words in the $T_i$ except for the longest word $T_{w_\circ}$.
The same turns out to hold also for the generalized nil-Coxeter algebra
$A := NC_{A_1^n}({\bf d})$, by defining a trace map to kill all words in
the $T_i$, except for $\prod_{i=1}^n T_i^{d_i - 1}$. Note that these two
words $T_{w_\circ}$ and $\prod_{i=1}^n T_i^{d_i - 1}$ span the space
$\Prim(A) = \Prim(A^{op})$, as we note after Theorem \ref{Tcenter} below.
%}}}

\section{Deformations over cocommutative algebras with nilpotent maximal
ideals}

In this final section, we study the representations of deformed smash
product algebras over nil-Coxeter algebras. We will work in somewhat
greater generality.

\begin{stand}\label{Assume2}
Henceforth, $\bk$ is a field, and $(A, \Delta)$ is a cocommutative
$\bk$-algebra with coproduct, with a nilpotent maximal ideal $\m = A \m A
\neq 0$ that satisfies:
\[
A = \m \oplus \bk \cdot 1_A, \qquad
\exists \ell_A \in \mathbb{N} : \m^{\ell_A} = 0 \neq \m^{\ell_A-1},
\qquad \Delta(\m) \subset \m \otimes \m.
\]
\end{stand}

\noindent We will use without further reference the following
observations, when required:
\begin{itemize}
\item $(A,\m)$ is local, since every element in $A \setminus \m$ is
invertible. From this one can show that $\m$ is the Jacobson radical of
$A$, and $\Ext_{A-mod}(\bk, \bk) \cong (\m^2)^\perp$, where $(\m^2)^\perp
\subset \m^*$.

\item The assumption $\Delta(\m) \subset \m \otimes \m$ is required if
$\ch \bk > 0$. Cocommutative algebras not satisfying this assumption
exist; for instance, consider $A := (\Z / p\Z)[T]
/ (T^p)$, with $p>0$ prime and $\Delta(T) = 1 \otimes T + T \otimes 1$. 
However, we do not need to assume $\Delta(\m) \subset \m \otimes \m$ if
$\ch \bk = 0$. Indeed, given $a \in \m$, let $\Delta(a) \in c (1 \otimes
1) \oplus d(1 \otimes \m) \oplus e(\m \otimes 1) \oplus (\m \otimes \m)$,
with $c,d,e \in \bk^\times$. By multiplicativity, $\Delta(a)^n = 0$ for
$n \gg 0$, which works out to: $c=d=e=0$.
\end{itemize}

The prototypical example of an algebra satisfying Assumption
\ref{Assume2} is the nil-Coxeter algebra $NC_W$ for a finite Coxeter
group $W$. Another example is the generalized nil-Coxeter algebra
$NC_{A_1^n}((d_1, \dots, d_n)) = \otimes_{i=1}^n \bk[T_i] / (T_i^{d_i})$.
In both cases, $\m$ is the two-sided augmentation ideal generated by the
$T_i$.
We remark for completeness that in related work \cite{Khnilcox}*{Theorem
C}, we characterize the generalized nil-Coxeter algebras $NC_W({\bf d})$
for which the maximal ideal $\m$ is nilpotent. This property turns out to
be equivalent to the finite-dimensionality of $NC_W({\bf d})$, which was
discussed following Remark \ref{Rweak}.

%{{{1 Section 6.1 - Simple $\hb$-modules
\subsection{Simple $\hb$-modules}

We begin by exploring simple modules over $\hb$. In order to state our
results, some notation is required.

\begin{definition}\label{Dlevel}
Suppose $A$ is as in Assumption \ref{Assume2}, and $M$ is an $A$-module.
\begin{enumerate}
\item The \textit{level} of a nonzero vector $m \in M$ is the integer $k
> 0$ such that $\m^k m = 0 \neq \m^{k-1} m$. Define the level of $0_M$ to
be $0$ for convention. The \textit{level of the module}, denoted by
$\ell_M$, is the highest level attained in $M$.

\item For $k \geqslant 0$, define $\lev{k}(M)$ to be the set of elements
of level at most $k$.

\item A vector $m \in M$ is \textit{primitive} if $\m m = 0$.
Let $\Prim(M)$ denote all primitive elements.
\end{enumerate}
\end{definition}

The following lemma is easily shown.

\begin{lemma}\label{Llevel}
Suppose $M$ is any $A$-module. Then $\lev{k}(M) = \ker_M \m^k$; in
particular,
\[
\Prim(M) = \lev{1}(M), \qquad
M = \lev{\ell_M}(M), \qquad
\ell_M \leqslant \ell_{A^{mult}} = \ell_A.
\]
Moreover, $\lev{k}(M)$ is a proper submodule of $\lev{k+1}(M)$ for all $k
< \ell_M$.
\end{lemma}

We now study $\hb$-modules. Our first result aims to classify all simple
$\hb$-modules in the case when $\kappa_V \equiv 0$.

\begin{theorem}\label{Tsimple}
Suppose $A$ satisfies Assumption \ref{Assume2} and $V$ is an $A$-module.
If $\lambda$ satisfies Equation \eqref{Eaction} in $A$, then
$\lambda(\m^k, \lev{k}(V)) \subset \m^k$ for all $k \geqslant 0$.
If instead we assume $\kappa_V \equiv 0$, then the following are
equivalent for $\hb$:
\begin{enumerate}
\item $\lambda(\m^k, V) \subset \m^k$ for all $k \geqslant 0$, and
$\kappa_A : V \wedge V \to \m$.

\item $\lambda(\m, V) \subset \m$ and $\kappa_A : V \wedge V \to \m$.

\item There exists a one-dimensional $\hb$-module killed by $\m$.

\item There is a bijection from simple $\hb$-modules to simple
$\Sym(V)$-modules, determined uniquely by restriction from $\hb$ to the
image of $V$; moreover, the inverse map is given by restriction to $V$
and inflation to $\hb$, letting $\m$ act trivially.
\end{enumerate}
\end{theorem}

\noindent The condition $\kappa_A : V \wedge V \to \m$ is a natural one
in characteristic zero, in the sense that it is necessary if $\hb$ has a
finite-dimensional module and $\ch \bk = 0$. This is because if $\pi :
\hb \to \End_\bk M$ is a finite-dimensional representation, then for all
$a \in \m$, $\pi(a)$ is nilpotent, hence has trace zero. It follows that
$\im \kappa_A = [V,V] \subset \m$.

The following result will be useful in proving Theorem \ref{Tsimple}.

\begin{prop}\label{P1}
Suppose $M$ is an $A$-module.
\begin{enumerate}
\item $M$ is $A$-semisimple if and only if $\m M = 0$.

\item Any finite filtration $M = M_0 \supset M_1 \supset \cdots \supset
M_k = 0$ of $A$-modules (such as $M \supset 0$) can be refined to a
possibly longer finite filtration, so that the successive subquotients
are $A$-semisimple modules. In particular, $\Prim(M) \neq 0$ if $M \neq
0$.

\item Every maximal submodule of a nonzero $A$-module has codimension
one. Thus a $d$-dimensional $A$-module has a flag of $A$-submodules of
length $d+1$.

\item $\Prim(A) \subset \m$.

\item If $M$ is nonzero, $\m M$ is contained in every maximal proper
(i.e. codimension one) submodule of $M$. In particular, it is a proper
submodule of $M$ if $M \neq 0$.
\end{enumerate}
\end{prop}

\begin{proof}\hfill
\begin{enumerate}
\item If $\m M = 0$ then $M$ is clearly $A$-semisimple. Conversely,
if $M$ is $A$-semisimple, notice that $M = \m M \oplus M_1$ for
some $A$-semisimple complement $M_1$. But then $M_1 \cong M / \m M$ is
annihilated by $\m$. Repeat this construction on $\m M$ to produce $M_2$,
and so on; this process stops after finitely many steps as $\m$ is
nilpotent. But then $M$ is a direct sum of submodules killed by $\m$.

\item It suffices to prove the result for the filtration $M \supset 0$.
Define $M_i := \m^i M$ for all $i > 0$, and $M_0 := M$. Now apply the
previous part.

\item This follows from the previous part.

\item If $a \in A \setminus \m$, then $a$ is invertible, hence cannot lie
in $\Prim(A)$.

\item Suppose $M = \bk m_0 \oplus M'$ where $M'$ is a proper submodule.
Fix $a \in \m$ such that $a m_0 = r m_0 + m'$, with $r \in \bk$
and $m' \in M'$. Then one shows by induction on $i$ that
\[
a^i m_0 = r^i m_0 + (r^{i-1} m' + r^{i-2} am' + \dots + a^{i-1}m')
\]

\noindent for all $i > 0$. In particular, since $a^{\ell_A} \in
\m^{\ell_A} = 0$, hence $r^{\ell_A} m_0 \in M'$, whence $r^{\ell_A} = 0$.
Thus $r=0$, and $am_0 = m' \in M'$ for all $a \in \m$, whence $\m M
\subset M'$ as claimed.
\end{enumerate}
\end{proof}

\begin{proof}[Proof of Theorem \ref{Tsimple}]
The first assertion holds because the $A$-action \eqref{Eaction} implies
that if $\m^k(v) = 0$, then (with a slight abuse of notation)
\[
0 = \lambda(\m^{\ell_A - k} \m^k, v) = \m^{\ell_A - k} \lambda(\m^k, v) +
\lambda(\m^{\ell_A - k}, \one{\m^k}(v)) \two{\m^k} = \m^{\ell_A - k}
\lambda(\m^k,v),
\]

\noindent from which it follows that $\lambda(\m^k,v) \subset \m^k$.

We now assume $\kappa_V \equiv 0$, and show that (1) and (2) are
equivalent. Clearly $(1) \implies (2)$; conversely, if (2) holds, then we
compute for $a_1, \dots, a_k \in \m$, by induction on $k$:
\begin{align*}
\lambda(a_1 \cdots a_k, v) = & a_1 \lambda(a_2 \cdots a_k, v) + \sum
\lambda(a_1, (\one{(a_2)} \cdots \one{(a_k)})(v)) \one{(a_2)} \cdots
\one{(a_k)}\\
\subset & \m \cdot \m^{k-1} + \m \cdot \m^{k-1} = \m^k.
\end{align*}

Next, given (2), we show (4) as follows: if $M$ is a simple
$\Sym(V)$-module then the construction in (4) makes it a simple
$\hb$-module, as the relations in $\hb$ indeed hold in $\End_\bk M$ via
(2). On the other hand, given any $\hb$-module $M$, by Proposition
\ref{P1}, $\ker_M \m \neq 0$. We now \textbf{claim} that if $\lambda(\m,
V) \subset \m$ and $M$ is a $\hb$-module, then $\ker_M \m^k$ is a
$\hb$-submodule of $M$. Given the claim, if $M$ is now a simple
$\hb$-module, then $0 \neq \ker_M \m$ is a $\hb$-submodule, whence $\m M
= 0$, proving (4).

It remains to show the claim (in order to complete the proof of $(2)
\implies (4)$).
Let $M' = \ker_M \m^k$; then for $a \in \m$ and $m' \in M'$, we have
\[
\m^k (am') \subset \m^k A \cdot m' = \m^k m' = 0,
\]

\noindent whence $am' \in M'$. Thus $M'$ is an $A$-submodule. It thus
remains to show that $vm' \in M'$ for $v \in V$. But if we have $a_1,
\dots, a_k \in \m$, then
\[
\prod_{i=1}^k a_i \cdot vm' = \sum \left( \prod_{i=1}^k \one{(a_i)}
\right) (v) \cdot \prod_{i=1}^k \two{(a_i)} \cdot m' + \lambda(a_1 \cdots
a_k, v) m',
\]

\noindent and this is killed by using Assumption \ref{Assume2} and the
equivalence of (1) and (2). Hence $vm' \in M'$.

Finally, we show $(4) \implies (3) \implies (2)$. If (4) holds, choose
any linear functional $\mu \in V^*$ and consider the simple
one-dimensional $\Sym(V)$-module
\[
M_\mu := \Sym(V) / \Sym(V) \cdot (\im (\id_V - \mu)).
\]
By (4), $M_\mu$ yields a one-dimensional simple $\hb$-module which is
killed by $\m$, and this shows (3). Next, if (3) holds for $M$ then $V$
acts on $M$ by scalars, i.e., by $\mu \in V^*$. It follows that $\im
\kappa_A = [V,V]$ kills $M$, whence $\kappa_A : V \wedge V \to \m$.
Similarly if $a \in \m$, then $\lambda(a,v) \in \m V - V \m$ also kills
$M$, whence $\lambda(\m,V) \subset \m$.
\end{proof}

\begin{cor}
Suppose $\bk$ is algebraically closed and $V$ is finite-dimensional. If
$\lambda(\m, V) \subset \m$, $\kappa_V \equiv 0$, and $\kappa_A : V
\wedge V \to \m$, then all simple finite-dimensional $\scrh_{\lambda,
\kappa_A}$-representations are one-dimensional, and in bijection with
$V^*$.
\end{cor}
%}}}

%{{{1 Section 6.2 - PBW property
\subsection{PBW property}

Our next goal is to prove a result similar to Theorem \ref{Tnilcox} that
classifies the PBW deformations $\hb$, but in the more general setting of
cocommutative algebras $A$ satisfying Assumption \ref{Assume2}. Thus we
do not assume the existence of a grouplike basis as for the nil-Coxeter
algebra, and alternate methods are required. In particular, the following
provides a second proof of Theorem \ref{Tnilcox}.

\begin{theorem}
Suppose $A$ satisfies Assumption \ref{Assume2}, and $V$ is an $A$-module.
\begin{enumerate}
\item Suppose $\kappa_V \equiv 0$. Then the Jacobi identity
\eqref{Ejacobi2} holds in $\scrh_{\lambda, \kappa_A}$ if and only if
$\dim_\bk V \leqslant 2$ or $\im \kappa_A \subset \bk \cdot 1_A$.

\item If $\dim_\bk V \leqslant 2$, then $\scrh_{0, \kappa}$ has the PBW
property if and only if $\im \kappa_V \subset \Prim(V)$ and $\im \kappa_A
\subset \Prim(A^{mult})$.

\item If $\dim_\bk V > 2$, and $\lambda, \kappa_V \equiv 0$, then
$\scrh_{0, \kappa_A}$ has the PBW property if and only if $\kappa_A
\equiv 0$.
\end{enumerate}
\end{theorem}

\begin{proof}\hfill
\begin{enumerate}
\item By Remark \ref{Rdim2}, and since $\kappa_V \equiv 0$, it suffices
to characterize the Jacobi identity \eqref{Ejacobi2} under the additional
assumption that $\dim V > 2$. Now write down the identity:
\[
\sum_\circlearrowright [\kappa(v_1,v_2),v_3] = 0, \qquad v_1, v_2, v_3
\in V.
\]

\noindent We may assume without loss of generality that the $v_i$ are
linearly independent in $V$. Moreover, the $\kappa_1$-component is killed
by commuting with elements of $V$. (Here, we work with a distinguished
$\bk$-basis of $\m$, along with $\{ 1_A \}$.) If we now define
$\gamma_{v,v'} := \kappa_A(v,v') - \kappa_1(v,v') \in \m$, then
\[
\sum_\circlearrowright \left( v_1 \gamma_{v_2,v_3} - \sum
\one{(\gamma_{v_2,v_3})}(v_1) \two{(\gamma_{v_2,v_3})} \right) = 0.
\]

\noindent Now assume without loss of generality that $v_1 \in
\lev{k+1}(V) \setminus \lev{k}(V)$ for some $k \geqslant 0$, and $v_1,
v_2, v_3 \in \lev{k+1}(V)$. Then $\one{(\gamma_{v_p,v_q})}(v_r) \in
\lev{k}(V)$ for all $\{ p,q,r \} = \{ 1, 2, 3 \}$. Working modulo
$\lev{k}(V)$, it follows by the linear independence of the $v_i$ that
$\gamma_{v_2,v_3} = 0$, and hence an entire summand in the above cyclic
sum vanishes. Repeat the same argument twice to show all summands are
zero, and hence, $\kappa_A \equiv \kappa_1$ on $V \wedge V$.

\item This is similar to the proof of Theorem \ref{Tnilcox}(1) and is
omitted for brevity.

\item Clearly $\scrh_{0,0}$ has the PBW property. Conversely, assume
$\scrh_{0,\kappa_A}$ has the PBW property. By a previous part, we have
$\im \kappa_A \subset \bk \cdot 1_A$. Suppose $\kappa_A \not\equiv 0$.
Then there exists $k \geqslant 0$ such that $\kappa_A(\lev{k+1}(V), V)
\not\equiv 0 = \kappa_A(\lev{k}(V),V)$. Choose nonzero $a \in \m$, and
any $v_0 \in \lev{k+1}(V)$, $v_1 \in V$ such that $\kappa_A(v_0, v_1)
\neq 0$. Then by Theorem \ref{Tpbw},
\[
0 \neq a \kappa_A(v_0, v_1) = \sum \kappa_A(\one{a}(v_0), \two{a}(v_1))
\three{a}.
\]

\noindent But by assumption $\one{a}(v_0) \in \lev{k}(V)$, whence the
right hand side vanishes. This contradiction shows that $\kappa_A \equiv
0$.
\end{enumerate}
\end{proof}
%}}}

%{{{1 Section 6.3 - Center and abelianization
\subsection{Center and abelianization}

We end the paper by computing the center and abelianization of the
algebra $\hb$, i.e., the zeroth Hochschild (co)homology.

\begin{theorem}\label{Tcenter}
Suppose $A$ satisfies Assumption \ref{Assume2}, $V, \lambda, \kappa$ are
such that $\hb$ has the PBW property, and $\Prim(A) = \Prim(A^{op})$.
If $\lambda(\m, V) \subset \m$, then $\hb$ has trivial center, i.e.,
$HH^0(\hb,\hb) = \bk$.
\end{theorem}

\noindent Akin to the remarks following Assumption \ref{Assume2}, the
condition $\Prim(A) = \Prim(A^{op})$ is satisfied by all nil-Coxeter
algebras $NC_W$ for a finite Coxeter group $W$, as well as by
$NC_{A_1^n}({\bf d})$. The condition $\lambda(\m,V) \subset \m$ was
discussed in detail in Theorem \ref{Tsimple}.

\begin{proof}
We first choose a totally ordered basis of $V$ as follows: via
Proposition \ref{P1}, fix the filtration $0 = \lev{0}(V) \subset
\lev{1}(v) \subset \cdots \subset \lev{\ell_V}(V) = V$ according to the
level; then choose any $\bk$-basis $\mathcal{B}_k$ of the corresponding
vector space complement of $\lev{k-1}(V)$ in $\lev{k}(V)$ for $k = 1,
\dots, \ell_V$. Now index $\mathcal{B}_k$ by any totally ordered set
$S_k$, and let $S := \bigsqcup_k S_k$ be totally ordered via: $s_i < s_j$
if $i>j$ and $s_i \in S_i, s_j \in S_j$. Thus, every element of
$\mathcal{B}_1$ is primitive. Now use the PBW property to write any
vector in $\hb$ as $\sum_I v_I a_I$, where $I$ denotes a word in $S$
whose letters occur in non-increasing order, $a_I \in A$, and $v_I$
denotes the corresponding monomial in $\bigsqcup_k \mathcal{B}_k$.

Note that $\m$ acts on each $v_I$ and yields a linear combination of
elements $v_J$ such that $I > J$ in the lexicographic order on words in
$S$. More precisely, if we define $\ell(v_I)$ to be the sum of the levels
of the letters in the monomial $v_I$ (see Definition \ref{Dlevel}), then
$\m$ strictly reduces $\ell(v_I)$.

We now proceed to the proof. Suppose $0 \neq z = \sum_I v_I a_I$ is
central in $\hb$, with the $v_I$ linearly independent. We first
\textbf{claim} that for each non-empty $I$, the vector $a_I$ is primitive
in $A$. Indeed, choosing $a \in \m$ and writing out $az = za$ yields:
\[
\sum_I \left(\sum \one{a}(v_I) \two{a} + \lambda(a, v_I) \right) a_I =
\sum v_I a_I a.
\]

\noindent Choosing $I \neq \emptyset$ such that $v_I$ has maximal
$\ell$-value, it follows from above that $a_I a = 0$ for all $a \in \m$.
Hence $a_I \in \Prim(A^{op}) = \Prim(A)$ by assumption. Now say $v_I =
v_{i_k} \cdots v_{i_1}$ for some $i_j \in I$. We notice by induction on
$k$ that $a v_I a_I = 0$ as well. Indeed,
\[
a v_I a_I = \sum \one{a}(v_{i_k}) \cdot \two{a} v_{i_{k-1}} \cdots
v_{i_1}  a_I + \lambda(a,v_{i_k}) \cdot v_{i_{k-1}} \cdots v_{i_1}  a_I,
\]

\noindent and both expressions vanish by the induction hypothesis (the
base case of $k=1$ is easy). It follows that $a v_I a_I = 0 = v_I a_I a$,
where $I \neq \emptyset$ is such that $\ell(v_I)$ is maximal. Now cancel
these terms from the above equation and work with $I$ of the next highest
$\ell$-value. Repeating the above analysis shows the claim.

Next, let $v \in \Prim(V)$ and consider $zv = vz$ in $\hb$:
\[
a_\emptyset v + \sum_I v_I a_I v = v a_\emptyset + \sum v v_I a_I.
\]

\noindent Since $a_I \in \Prim(A) \subset \m$ (by Proposition \ref{P1}),
hence $a_I v = \lambda(a_I, v)$ for all non-empty $I$. Hence working
modulo the filtered degree $\leqslant 1$ piece and using the PBW
property, $a_I = 0$ if $I \neq \emptyset$.
In other words, $z = a_\emptyset \in A$. Since $A = \bk \cdot 1 \oplus
\m$, we may assume that $z \in \m$. Now choose nonzero primitive $v \in
V$; then,
\[
v z = z v = \sum \one{z}(v) \two{z} + \lambda(z,v) = \lambda(z,v),
\]

\noindent whence we get that $z = 0$ by the PBW property. Hence $Z(\hb) =
\bk \cdot 1$ as claimed.
\end{proof}

Next, we compute the zeroth Hochschild homology.

\begin{theorem}
Suppose $\lambda$ and $\kappa_V$ are identically zero, $\kappa_A : V
\wedge V \to \m$, and $\scrh_{0, \kappa_A}$ satisfies the PBW property.
If $\bk$ is an infinite field, then as abelian $\bk$-algebras, we have
\begin{align*}
HH_0(\hb,\hb) = &\ \frac{\hb}{[\hb,\hb]}\\
\cong &\ \bk \cdot 1 + \left( \Sym^+(V) \bigoplus \left( \m / ([\m,\m] +
A \cdot (\im \kappa_A) \cdot A) \right) \right),
\end{align*}

\noindent where the direct sum indicates that the two factors are ideals
and hence multiply to zero.
\end{theorem}

\begin{proof}
The proof is in steps. The first step is to show that $[\hb,\hb]$
contains the image of $\tangle{V} \cdot \m$, where given a subspace $U
\subset V,\ \tangle{U} := TV \cdot U \cdot TV$ is the two-sided ideal in
$TV$ generated by $U$. More precisely, we show by induction on $k$ that
$\tangle{\lev{k}(V)} \cdot \m \subset [\hb, \hb]$. This is clear for
$k=0$, and given the result for $k$,
Assumption \ref{Assume2} implies that
\[
a(p) \in \tangle{\lev{k}(V)}, \qquad \forall a \in \m,\ p \in
\tangle{\lev{k+1}(V)}.
\]

\noindent It follows by the induction hypothesis that
\begin{align*}
p \cdot a = &\ [p,a] + a \cdot p\\
= &\ [p,a] + \sum \one{a}(p) \two{a} \in [\hb, \hb] + \tangle{\lev{k}(V)}
\m \subset [\hb, \hb].
\end{align*}

Next, fix a total ordering on a basis of $V$. Given any nonzero sum ${\bf
v}$ of monomial ``ordered'' words, since $\bk$ is an infinite field there
exists $\mu \in V^*$ such that $\mu({\bf v}) \neq 0$. Now since $\lambda
\equiv 0$, it follows by Theorem \ref{Tsimple} that $\hb$ has a
one-dimensional representation $M_\mu$ killed by $\m$, and on which $V$
acts by $\mu$. Since $[\hb, \hb]$ necessarily kills $M_\mu$, it follows
that ${\bf v}$ has nonzero image in $\hb / [\hb, \hb]$. Hence $V$
generates the symmetric algebra in $\hb / [\hb, \hb]$.

It remains to consider the image of $A$ inside the abelianization.
Note that $\im \kappa_A = [V,V]$ and $[\m,\m]$ lie in $[\hb, \hb]$,
and are subspaces of $\m$ by assumption. (That this image and $\Sym^+(V)$
are ideals follows from the above analysis.)
To complete the proof, it suffices to show the commutator intersects
$A$ in $[\m,\m] + A \cdot (\im \kappa_A) \cdot A$. Note 
$\hb = A \bigoplus \tangle{V} \cdot A$ by the PBW property. Now
$[A,A] = [\m,\m]$, while $[\tangle{V} \cdot A, A] \subset \tangle{V}
\cdot A$, which intersects $A$ trivially.

It remains to consider $[\tangle{V} \cdot A, \tangle{V} \cdot A] \cap A$.
By the relations in $\hb$ as well as the PBW property, the only elements
that occur here arise from the relations $[v,v'] = \kappa_A(v,v') \in A$,
and hence the intersection is contained in $A \cdot (\im \kappa_A) \cdot
A$. We now show that this containment is an equality, via the
\textbf{claim} that $a \kappa_A(v,v') a' \in [\hb, \hb]$ for  $v,v' \in
V$ and $a,a' \in A$.
The claim is obvious if $a=a'=1$. Otherwise we may assume that at least
one of $a,a'$ lies in $\m$. In this case,
\begin{align*}
[av,v'a'] = &\ avv'a' - v'a'av = a[v,v']a' + av'va' - v'a'av\\
= &\ a[v,v']a' + \sum \one{a}(v') \two{a}(v) \three{a} a' - v' \sum
\one{(a'a)}(v) \two{(a'a)}.
\end{align*}

\noindent Since $\Delta(\m) \subset \m \otimes \m$, it follows that all
summands of both sums lie in  $\tangle{V} \cdot \m$, hence in $[\hb,
\hb]$ from above. This proves the claim, and with it, the result.
\end{proof}
%}}}

\section*{Acknowledgments}
The author would like to thank Sarah Witherspoon for many stimulating and
informative conversations regarding this paper. The author also thanks
Ivan Marin, Susan Montgomery, and Victor Reiner for useful references and
discussions. Finally, the author is grateful to Chelsea Walton for going
through a preliminary draft of this work and for her helpful suggestions.

%{{{1 Bibliography
\bibliographystyle{amsplain}

%}}}

\end{document}